\documentclass[11pt]{amsart}

\usepackage{amssymb,latexsym}

\usepackage{graphicx,epsfig}
\usepackage[dvipsnames]{xcolor}

\def\R{{\mathbb R}}	
\def\E{{\mathcal E_\gamma}}  
\def\T{{\mathbb T}}
\def\N{{\mathbb N}}

\def\Z{{\Bbb Z}}
\def\F{{\mathcal F}}	
\def \supp {\text{\rm supp\,}}

\def\trans{{\,}^t}

\newtheorem{thmnr}{Theorem}[section]
\newtheorem{lemnr}[thmnr]{Lemma}

\newtheorem{pronr}[thmnr]{Proposition}
\newtheorem{remark}[thmnr]{Remark}
\newtheorem{remarks}[thmnr]{Remarks}

\newtheorem{defns}{Definitions}[section]

\newcommand{\secret}[1]{}

\begin{document}

\title[Fourier restriction for negative curvature: polynomial partitioning]{A Fourier
restriction
theorem for a perturbed hyperbolic paraboloid: polynomial partitioning}

\author[S. Buschenhenke]{Stefan Buschenhenke}
\address{S. Buschenhenke:  Mathematisches Seminar, C.A.-Universit\"at Kiel,
Ludewig-Meyn-Stra\ss{}e 4, D-24118 Kiel, Germany}
\email{{\tt buschenhenke@math.uni-kiel.de}}
\urladdr{http://www.math.uni-kiel.de/analysis/de/buschenhenke}

\author[D. M\"uller]{Detlef M\"uller}
\address{D. M\"uller: Mathematisches Seminar, C.A.-Universit\"at Kiel,
Ludewig-Meyn-Stra\ss{}e 4, D-24118 Kiel, Germany}
\email{{\tt mueller@math.uni-kiel.de}}
\urladdr{http://www.math.uni-kiel.de/analysis/de/mueller}

\author[A. Vargas]{Ana Vargas}
\address{A. Vargas: Departmento de Mathem\'aticas, Universidad Aut\'onoma de  Madrid,
28049
Madrid,
Spain
}
\email{{\tt ana.vargas@uam.es}}
\urladdr{http://matematicas.uam.es/~AFA/}

\thanks{The first author was partially supported by the ERC grant 307617.\\
The first two authors were partially supported by the DFG grants MU 761/ 11-1 and MU 761/
11-2.\\
The third author was partially supported grant
MTM2016-76566-P, Ministerio de Ciencia, Innovaci\'on y Universidades (Spain).}

\begin{abstract} We consider a surface with negative curvature in $\R^3$ which  is a cubic
perturbation of the
saddle. For this surface, we prove a new restriction theorem, analogous to the theorem for
paraboloids proved by
L. Guth  in 2016. 
 This specific perturbation has turned out to be of
  fundamental importance also to the understanding of more general classes of
  perturbations.
\end{abstract}

\maketitle

\tableofcontents

\thispagestyle{empty}

\setcounter{equation}{0}
\section{Introduction}\label{intro}

The Fourier restriction problem, introduced by E. M. Stein in the seventies (for general submanifolds), asks for  the range  of exponents
$\tilde p$ and $\tilde q$ for which  an a priori  estimate of the form
\begin{align*}
\bigg(\int_S|\widehat{f}|^{\tilde q}\,d\sigma\bigg)^{1/\tilde q}\le C\|f\|_{L^{\tilde
p}(\R^n)}
\end{align*}
holds  true for   every Schwartz function   $f\in\mathcal S(\R^3),$ with a constant $C$
independent of $f.$ Here, $d\sigma$ denotes the surface measure on $S.$

The sharp range in dimension  $n=2$ for curves with non-vanishing curvature was determined
through work by  C. Fefferman, E. M. Stein and A. Zygmund \cite{F1}, \cite{Z}. In higher
dimension, the sharp $L^{\tilde p}-L^2$ result  for hypersurfaces with
non-vanishing Gaussian curvature was obtained by E. M. Stein and P. A. Tomas \cite{To},
\cite{St1} (see also Strichartz \cite{Str}). Some more general classes of surfaces were
treated by A. Greenleaf \cite{Gr}. In work by I. Ikromov, M.
Kempe and D. M\"uller  \cite{ikm}  and Ikromov and M\"uller \cite{IM-uniform}, \cite{IM},
the sharp range of Stein-Tomas type    $L^{\tilde p}-L^2$  restriction
estimates has been  determined  for a large class  of smooth, finite-type hypersurfaces,
including all analytic hypersurfaces.

The question about  general  $L^{\tilde p}-L^{\tilde q}$ restriction
estimates is nevertheless still wide open. Fourier restriction to hypersurfaces  with non-negative principal curvatures has been studied intensively by many authors. Major progress was due to J. Bourgain in the nineties (\cite{Bo1}, \cite{Bo2}, \cite{Bo3}). At the end of that decade the bilinear method was introduced (\cite{MVV1}, \cite{MVV2}, \cite{TVV} \cite{TV1}, \cite{TV2},  \cite{W2}, \cite{T2},  \cite{lv10}). A new impulse to the problem has been given with  the multilinear method (\cite{BCT}, \cite{BoG}). The best results up to date have been obtained with the polynomial partitioning method, developed by L. Guth (\cite{Gu16},
\cite{Gu17}) (see also \cite{hr19} and \cite{WA18} for recent improvements).

For the case of hypersurfaces of  non-vanishing  Gaussian curvature but principal curvatures of different signs, besides Tomas-Stein type Fourier restriction estimates, until recently  the only case which had been studied successfully was the case of the hyperbolic paraboloid (or ``saddle") in  $\R^3$:
in 2015, independently S. Lee \cite{lee05} and A. Vargas \cite{v05}  established results analogous to Tao's theorem \cite{T2} on elliptic surfaces (such as the $2$ -sphere), with the exception of the  end-point, by means of the bilinear method.  Recently, B. Stovall
\cite{Sto17}  was able to include also the end-point case. Moreover,  C. H. Cho and J. Lee \cite{chl17}, and
J. Kim \cite{k17},  improved the range by adapting ideas by Guth  \cite{Gu16},
\cite{Gu17} which are based on the polynomial partitioning method.  Results on higher dimensional hyperbolic paraboloids  have just been reported  by A. Barron \cite{ba20}.

\medskip

In our previous paper \cite{bmv18}, we considered a one variable perturbation of the hyperbolic paraboloid, and applied the bilinear method, obtaining results analogous to \cite{lee05}  and \cite{v05}.  Further results for more general classes of one-variate finite type,  respectively flat,  perturbations based on the bilinear method were obtained in \cite{bmv19}, \cite{bmv20}. Bilinear estimates are also key elements in the results  obtained with the polynomial partitioning method for the non--negative curvature case. With the base of our previous bilinear results, we explore in this article the application of that method to our model surfaces. We obtain the analogous result to \cite{Gu16} for our class of  hyperbolic surfaces.

\medskip
More precisely, we consider the family of functions
$$
\phi_\gamma(x,y)=xy+\frac\gamma 3y^3\qquad\text{for } -1\le \gamma\le1,
$$
defined on   $\Sigma:=[0,1]\times [0,1],$ and the corresponding surfaces
$$
\bf S_\gamma=\{(x,y,\phi_\gamma(x,y)):  (x,y)\in \Sigma\}.
$$
The associated adjoints to the corresponding Fourier restriction operators are the {\it extension} operators   given by
$$
\E(\xi):=\int_{\Sigma} f(x,y)\,e^{i[\xi_1x+\xi_2y+\xi_3\phi_\gamma(x,y)]}\,dx\,dy,\qquad \xi=(\xi_1,\xi_2,\xi_3)\in\R^3.
$$

Our main result will be the following analogue of   a result by Bassam Shayya \cite{s17}  (see also Jongchon Kim \cite{k17}) for the unperturbed hyperbolic paraboloid:
\begin{thmnr}\label{sharp}
For any $p>3.25$ with  $p>2q',$  there is a constant $C_{p,q}$ which is independent of  $\gamma\in [-1,1]$ such that
\begin{align*}
		\|\E f\|_{L^{p}(\R^3)} \leq C_{p,q}
\|f\|_{L^{q}(\Sigma)}
\end{align*}
for all $f\in L^{q}(\Sigma)$.
\end{thmnr}
\begin{remarks}\nonumber
 \begin{itemize}
\item[(i)] Note that in this  result and the corresponding Fourier restriction estimate we can replace the domain $\Sigma:=[0,1]\times [0,1]$ by the larger neighborhood $[-1,1]\times [-1,1]$ of the origin, simply by dividing the latter into four sectors of angle $\pi/2$ and reducing the corresponding estimates in each of these sectors to the estimate given in the theorem by means of symmetry considerations.
\item[ii)] Our arguments in this paper easily extend to more general perturbations of $xy$  of cubic type in the sense of \cite{bmv19} in place of the perturbation $\frac \gamma 3 y^3,$  and  the same reasoning as in \cite{bmv19} then allows to prove  Fourier restriction to surfaces  given as the graph of   $\phi(x,y):=xy+h(y),$ where the function  $h$ is smooth and of finite type at the origin,  in the same range of $p$'s and $q$'s as in Theorem \ref{sharp}.
 \end{itemize}
\end{remarks}

\medskip
To simplify the understanding of this paper, we will closely follow the notation and structure of the paper \cite{Gu16}, which makes use of induction on scales arguments.

\smallskip
Denote by $B_R$ the cube $B_R:=[-R,R]^3, \, R\ge 0.$
For technical reasons that will become clear soon we  shall not be able to  induct on an $L^\infty\to L^{3.25}$ estimate  for $\E$ as  in \cite{Gu16} (Theorem 2.2). Instead, we shall induct on the following statement:

\begin{thmnr}\label{mainresult}
For any $\epsilon>0,$  there is a constant $C_\epsilon$  such that  for any $\gamma\in [-1,1]$ and for any  $R\ge1$
\begin{align*}
		\|\E f\|_{L^{3.25}(B_R)} \leq C_\epsilon R^\epsilon \|f\|_{L^2(\Sigma )}^{2/q}\, \|f\|_{L^\infty(\Sigma)}^{1-2/q},
\end{align*}
for all $3.25\geq q >2.6$ and all $f\in L^\infty(\Sigma)$.
\end{thmnr}
Applying this estimate to characteristic functions, we obtain the estimate
\begin{align*}
		\|\E f\|_{L^{3.25}(B_R)} \leq C_\epsilon R^\epsilon \|f\|_{L^{q,1}(\Sigma )},
\end{align*}
for all $q >2.6.$ Real interpolation with the trivial $L^1\to L^\infty$ estimate for the extension operator then  gives
\begin{align*}
		\|\E f\|_{L^{p}(B_R)} \leq C_\epsilon R^\epsilon \|f\|_{L^{q}(\Sigma )},
\end{align*}
for all $p>3.25,$ $p>2q'.$ Finally, an $\epsilon$-removal theorem (Theorem 5.3 in \cite{k17}) gives Theorem \ref{sharp}.

\section{Broad points}\label{broadpoints}

\begin{defns}\label{strip-broad}{\rm
Fix $K\gg 1$ to be a large dyadic number.
We introduce four different partitions of  the square $\Sigma=[0,1]\times[0,1]:$
\smallskip

We divide $\Sigma$ into $K^2$ disjoint squares (called {\it caps}) $\tau$ of sidelength $K^{-1}.$ For a cap $\tau,$ we denote by  $f_\tau:=f\chi_\tau.$ This  {\it basic decomposition} into caps  will play a fundamental role  in many places of our subsequent arguments, as in \cite{Gu16}. However, in contrast to \cite{Gu16}, it will play no role in the definition of $\alpha$-broadness  given below. For the latter  notion, the  next three decompositions will be relevant:
\smallskip

We divide $\Sigma $ into $K^{1/4}$ disjoint {\it long horizontal strips} $L$ of dimensions $1\times K^{-1/4},$
we divide $\Sigma $ into $K^{1/2}$ disjoint {\it long vertical strips} $L$ of dimensions
$K^{-1/2}\times 1$ and, finally, we divide $\Sigma $ into $K^{3/4}$ disjoint {\it short vertical strips} $L$ of dimensions
$K^{-1/2}\times K^{-1/4},$   by looking at all intersections of a long horizontal with a long  vertical strip. For a strip $L,$ we denote by  $f_L:=f\chi_L.$
\smallskip

Let $\alpha\in(0,1).$
Given the function $f,$ $\gamma\in [-1,1]$  and $K,$ we say that the point $\xi\in\R^3$ is {\it $\alpha$-broad for}
$\E f$ if
$$
 \max_L |\E f_L(\xi)|\le\alpha|\E f(\xi)|,
$$
where the $\max_L$ is taken over all
 \begin{itemize}
\item[a)] horizontal strips as above if $|\gamma| K^{1/2}\ge 1,$ or
\item[b)]  horizontal and vertical strips as above if $|\gamma| K^{1/2}< 1.$
 \end{itemize}
We define $Br_\alpha\E f(\xi)$ to be $|\E f(\xi)|$ if $\xi$ is $\alpha$-broad, and zero
otherwise.}
\end{defns}

\noindent{\bf Note:}
In contrast to \cite{Gu16}, we shall here consider the functions $f$ to be defined on the square $\Sigma ,$ which will have slight technical advantages, whereas Guth views them as functions on the surface $S.$ Of course, we can as well identify our functions $f$  with the corresponding  functions $(x,y,\phi_\gamma(x,y))\mapsto f(x,y)$ on ${\bf S_\gamma}.$ Accordingly, one can identify  our ``caps''  $\tau$ and strips $L$ with the corresponding subsets of the surface ${\bf S_\gamma}$ that are the graphs of $\phi_\gamma$  over these sets. This explain why we still like to call the sets $\tau$ ``caps'' .

\smallskip
We will prove the following analogue to Theorem 2.4. in \cite{Gu16}:

\begin{thmnr}\label{broadtheorem}
For any $0<\epsilon< 10^{-10},$  there are constants $K=K(\epsilon)\gg 1$ and $C_\epsilon$ such that for
any radius $R\ge1$ and for any $|\gamma|\le1$ 	
\begin{align*}
		\|Br_{K^{-\epsilon}}
\E f\|_{L^{3.25}(B_R)} \leq C_\epsilon R^\epsilon
\|f\|_{L^2(\Sigma )}^{12/13}\, \|f\|_{L^\infty(\Sigma )}^{1/13}
\end{align*}
for all $f\in L^\infty(\Sigma ).$ Moreover $K(\epsilon)\rightarrow\infty$ as
$\epsilon\rightarrow0.$
\end{thmnr}

Note that Theorem \ref{mainresult} follows from this theorem by
arguments that are similar to those in \cite{Gu16}. To show this, let us put  $p:=3.25.$
\smallskip

We  divide the domain of integration $B_R$ in \eqref{mainresult}   into   four subsets:
\begin{eqnarray*}
A&:=&\{\xi\in B_R:\xi\text{ is $K^{-\epsilon}$-broad for $\E f$}\},\\
B&:=&\{\xi\in B_R:|\E f_{L}(\xi)|> K^{-\epsilon}|\E f(\xi)|\text{ for some long horizontal strip }L\}, \\
C&:=&\{\xi\in B_R\setminus B:|\E f_{L}(\xi)|> K^{-\epsilon}|\E f(\xi)|\text{ for some long vertical strip }L\}
\\
D&:=&\{\xi\in B_R\setminus(B\cup C):|\E f_{L}(\xi)|> K^{-\epsilon}|\E f(\xi)|\text{ for some short vertical strip }L\}.
\end{eqnarray*}
By the definition of broad points, $B_R=A\cup B\cup C\cup D$. Notice also that if $|\gamma| K^{1/2}\ge 1$, then $C=D=\emptyset$ by construction.
\smallskip

If $\xi\in A$, then $|\E f(\xi)|=Br_{K^{-\epsilon}}\E f(\xi)$, so that the contribution of $A$ can be controlled using Theorem \ref{broadtheorem}. Notice that
$$\|f\|_{L^2(\Sigma )}^{12/13}\, \|f\|_{L^\infty(\Sigma )}^{1/13}\leq\|f\|_{L^2(\Sigma )}^{2/q}\, \|f\|_{L^\infty(\Sigma )}^{1-2/q},$$
since $q>2.6>13/6$.

For the other parts, we induct on the size of $R$.

\smallskip
For $\xi\in B$ we estimate
\begin{equation}\label{egamest}
|\E f(\xi)|< K^\epsilon \sup_L|\E f_{L}(\xi)| \leq K^\epsilon \Big(\sum_L|\E f_L(\xi)|^p\Big)^{1/p},
\end{equation}
where here the supremum and sum are taken over all long horizontal strips $L.$

 If $L=[0,1]\times [b, b+K^{-1/4}] $ is any of these long horizontal strips,
we scale and translate $y=b+K^{-1/4}y'.$ Then
\begin{align*}
	K^{1/4}\phi_\gamma(x,y)
	= (x+\gamma K^{-1/4}by')y'+\frac\gamma{3K^{1/2}}y'^3+K^{1/4} bx +b^2 \gamma y'+\text{ constant}.
\end{align*}
By applying the linear change of coordinates  $x'=x+\gamma K^{-1/4}by',$ we obtain
\begin{align*}
	K^{1/4}\phi_\gamma(x,y)=\phi_{\gamma/K^{1/2}}(x',y')+K^{1/4} bx'+\text{ constant}.
\end{align*}
Then
\begin{equation}\label{egamest2}
	|\E f_{L}(\xi)|= K^{-1/4}\big|\mathcal{E}_{\gamma/K^{1/2}}f^L\big(\xi_1+b\xi_3,(\xi_2-b\gamma \xi_1) K^{-1/4},\xi_3 K^{-1/4}\big)\big|,
\end{equation}
where we have defined $f^L$ by  $f^L(x',y'):=f_L(x,y)$, so that $\|f^L\|_2= K^{1/8}\|f_L\|_2$ and $\|f^L\|_\infty\leq\|f\|_\infty$. Note that we have $y'\in[0,1]$ and $x'\in[-1,2],$ since $|\gamma bK^{-1/4}|\leq1,$ and  that the function $\mathcal{E}_{\gamma/K^{1/2}}f^L$ is supported in a box  of dimensions $2R\times\frac{2R}{K^{1/4}}\times\frac{2R}{K^{1/4}}.$ What is crucial here  is  that, compared to $B_R,$ this box is shorter by the factor $2 K^{-1/4}\le  1/2$ in  the $\xi_3$-direction,  for $K$ sufficiently large.

A problem more of technical nature is that in $\xi_1$-direction it is still of  the same size as  $R.$ However, as we shall show in  Lemma \ref{aux}, we can automatically  pass from an estimate on a box $B_{R'}$ to a corresponding  estimate on the whole ``plate''   $P_{R'}:=\R^2\times [0,R']$ containing   $B_{R'}.$  Applying this in  the present situation, with
$R':=2 K^{-1/4}R\le R/2,$ by our induction hypothesis we may then assume that the following estimate holds true:
\begin{align*}
	\|\mathcal{E}_{\gamma/K^{1/2}}f^L\|_{L^{3.25}(P_{R'})} \leq C_\epsilon R'^\epsilon \|f^L\|_{L^2(\Sigma )}^{2/q}\,\|f^L\|_{L^\infty(\Sigma)}^{1-2/q}.
\end{align*}

Thus, by \eqref{egamest} and \eqref{egamest2}, we see that
\begin{align*}
	\|\E f\|_{L^p(B)}
	\leq&	K^{1/2p-1/4+\epsilon} \Big(\sum_L\|\mathcal{E}_{\gamma/K^{1/2}} f^L\|^p_{L^p\big(P_{2R/K^{1/4}}\big)}\Big)^{1/p}\\
	\leq& CC_{\epsilon} R^\epsilon K^{1/2p-1/4q'+3\epsilon/4} \|f\|_{2}^{2/q}\,\|f\|_{\infty}^{1-2/q}\\
	\leq& \frac 1{10}C_{\epsilon} R^\epsilon \|f\|_{2}^{2/q}\,\|f\|_{\infty}^{1-2/q},
\end{align*}
since $p> 2q'$.

\smallskip
For $\xi\in C$, i.e., in the case of  long vertical strips, we need to be a bit more careful.
The natural change of coordinates  is now $x=a+K^{-1/2}x'$,  if the long vertical strip $L$ is given by $L=[a, a+K^{-1/2}]\times [0,1].$  Then
\begin{align*}
	K^{1/2}\phi_\gamma(x,y)
	=& x'y+K^{1/2}\frac\gamma{3}y^3 +aK^{1/2}y= \phi_{\gamma K^{1/2}}(x',y)+aK^{1/2}y,
\end{align*}
so to fit into our scheme, we need that $|\gamma K^{1/2}|\leq 1$. This is the
reason why we consider this type of strips only when $|\gamma K^{1/2}|\le1.$
Then we find that
\begin{align*}
	|\E f_{L}(\xi)|= K^{-1/2}\big|\mathcal{E}_{\gamma K^{1/2}}f^L\big(\xi_1K^{-1/2},\xi_2+a\xi_3,\xi_3 K^{-1/2}\big)\big|,
\end{align*}
where $f^L$ is now defined by  $f^L(x',y):=f_L(x,y),$
and can argue in a similar way as in the preceding case.

\smallskip

As for $D,$  if $L=[a,a+K^{-1/2}]\times [b, b+K^{-1/4}] $ is any of the short vertical strips, then
we scale and translate $x=a+K^{-1/2}x',$ $y=b+K^{-1/4}y'.$ Then
\begin{align*}
	K^{3/4}\phi_\gamma(x,y)
	= (x'+\gamma K^{1/4}by')(y'+K^{1/4}b)+\frac\gamma{3}y'^3+K^{1/2} ay' +\text{ constant}.
\end{align*}
By applying the linear change of coordinates  $x''=x'+\gamma K^{1/4}by',$ $y''=y'$ (note that, since $|\gamma|K^{1/2}\le1,$ we have that $|\gamma|bK^{1/4}\le1$), we obtain
\begin{align*}
	K^{3/4}\phi_\gamma(x,y)=\phi_{\gamma}(x'',y'')+K^{1/4} bx''+K^{1/2} ay''+\text{ constant}.
\end{align*}
Then, if $\xi\in E,$
\begin{align*}
	|\E f_{L}(\xi)|= K^{-3/4}\big|\mathcal{E}_{\gamma}f^L\big(K^{-1/2}(\xi_1+b\xi_3),K^{-1/4}(\xi_2+a\xi_3-b\gamma \xi_1) , K^{-3/4}\xi_3\big)\big|,
\end{align*}
where we have defined $f^L$ by  $f^L(x'',y''):=f_L(x,y).$ From here on , we argue in a similar way as before.
\qed

\section{Reduction of Theorem \ref{broadtheorem} to a setup  allowing for inductive arguments}\label{reduction}

Following Section 3 in \cite{Gu16}, we shall next devise a setup and formulate a more general statement in Theorem \ref{largetheorem} which will become amenable to inductive arguments.
As in that paper, we change and extend our previous  notation slightly. We introduce a ``multiplicity''  $\mu\ge1,$  and choose accordingly   {\it caps }$\tau$ which now are allowed to be squares  of possibly larger side length   $r_\tau\in [K^{-1},\mu^{1/2}K^{-1}]$ than before. It can then happen that such a cap $\tau$  is no longer  contained in $\Sigma;$  in that case, we  truncate it by replacing it with its  intersection with $\Sigma.$

 We assume that we are given a family of such caps  $\tau$ covering $\Sigma =[0,1]\times[0,1]$ such that their centers are $K^{-1}$-
separated. Hence,  at any point there will be at most  $\mu$  of these caps which overlap at that point. Notice also that there are at most $K^2$
 caps $\tau$  in the family. We also assume that we have a decomposition
\begin{equation}\label{fsumtau}
 f=\sum_\tau f_\tau,
\end{equation}
 where
$\supp f_\tau\subset\tau.$

\smallskip

Given the family of caps, we define recursively a fixed family of {\it ragged long horizontal strips}  $S_\ell, \,(\ell=1,2,\dots,[\mu^{-1/2}K^{1/4}]),$ of
``widths" $\sim \mu^{1/2} K^{-1/4},$ in the following way:

\begin{eqnarray*}
 \F_1&:=&\{\tau: \tau^0\cap([0,1]\times[0,\mu^{1/2}K^{-1/4}])\ne\emptyset\}\qquad\text{and}\quad S_1:=\bigcup_{\tau\in\mathcal F_1}\tau,\\
\F_2&:=&\{\tau\notin\mathcal F_1: \tau^0\cap([0,1]\times
[\mu^{1/2}K^{-1/4},2\mu^{1/2}K^{-1/4}])\ne\emptyset\} \quad\text{and} \quad S_2:=\bigcup_{\tau\in\mathcal F_2}\tau,\\
&\vdots &\\
\F_\ell&:=&\{\tau \notin\mathcal\cup_{j=1}^{\ell-1}\mathcal F_j: \tau^0\cap([0,1]\times
[(\ell-1)\mu^{1/2}K^{-1/4},\ell\mu^{1/2} K^{-1/4}])\ne\emptyset\}\quad\text{and} \quad S_\ell:=\bigcup_{\tau\in\mathcal F_\ell}\tau,\\
&\vdots &
\end{eqnarray*}
Here, $\tau^0$ denotes the open interior of $\tau.$
Note that the families $\mathcal F_\ell$ are pairwise disjoint. Define $f_{S_\ell}:=\sum_{\tau\in\mathcal F_\ell} f_\tau,$ so that
$f=\sum_{\ell}f_{S_\ell}.$

\medskip

When $|\gamma| K^{1/2}\le1,$ we also define a family of  pairwise in measure disjoint {\it ragged long vertical strips} of  ``widths" $\sim \mu^{1/2} K^{-1/2}$   in an  analogous way, and a family of  pairwise in measure disjoint {\it ragged short  vertical strips} of  dimensions $\sim \mu^{1/2} K^{-1/2}\times \mu^{1/2}K^{-1/4}$ given by all intersections of a long horizontal and a long vertical strip,   and  add them to our set of ragged strips by denoting them by $S_\ell, \, \ell=[\mu^{-1/2} K^{1/4}]+1,\dots,$  and put as before $f_{S_\ell}:=\sum_{\tau\in\mathcal F_\ell} f_\tau.$

\medskip

Given a family of caps $\tau$  as above, and given the corresponding ragged strips $S_\ell$ and functions $f_\tau$ and
$f_{S_\ell}$ as before, we say that a point $\xi\in\R^3$ is $\alpha${\it -broad
for $\E f$ and the given family of caps},  if
$$
 \max_{S_\ell} |\E f_{S_\ell}(\xi)|\le\alpha|\E f(\xi)|,
$$
where the maximum is taken over the set of  all  ragged strips  $S_\ell$ as defined above (recall that this set  depends on the size of $|\gamma|K^{1/2}$).

We also define $Br_\alpha\E f(\xi):= |\E f(\xi)|$ if $\xi$ is $\alpha$-broad, and zero otherwise.
\smallskip

\begin{remark}\label{onbroad}
Note that when $\mu=1,$ then ragged strips are indeed strips in the sense of Definitions \ref{strip-broad}, and our present definition of broadness of points coincides in this case  with the one given before.
\end{remark}

The key result will be the following analogue to  Theorem 3.1 in \cite{Gu16}:

\begin{thmnr}\label{largetheorem}
For any $0<\epsilon< 10^{-10},$ there are constants $K=K(\epsilon)$ and $C_\epsilon,$ independent of
$\gamma\in [-1,1],$  such that for any family of caps $\tau$  with multiplicity at most $\mu$ covering $\Sigma$ as above and the associated family of
ragged strips $S_\ell$  and associated  functions $f_\tau$ and $f_{S_\ell}$   as defined  above which decompose $f,$  for
any length $R\ge1,$ any $\alpha\ge K^{-\epsilon}$ and for any $\gamma\in [-1,1],$ the following
holds true:

If for every $\omega\in \Sigma  ,$ and every cap $\tau$  as above,
\begin{equation}\label{average}
\oint_{B(\omega,R^{-1/2})}|f_\tau|^2\le 1,
\end{equation}

then,
\begin{equation}\label{broadest}
\int_{B_R}(Br_{\alpha} \E f)^{3.25} \leq C_\epsilon R^\epsilon
\bigg(\sum_\tau\int|f_\tau|^2\bigg)^{3/2+\epsilon}R^{\delta_{trans}\log(2K^\epsilon\alpha\mu)},
\end{equation}
where $\delta_{trans}:=\epsilon^6.$ Moreover
$K(\epsilon)\rightarrow\infty$ as $\epsilon\rightarrow0.$
\end{thmnr}
Here, in $\R^n,$ by $B(\omega,r)$ we denote the Euclidean ball of radius $r>0$ and center $\omega,$ and by $\oint_A f:=\frac 1{|A|}\int_A f$ we
denote the mean value $f$ over the measurable set $A$ of volume $|A|>0.$
\medskip

We can easily recover Theorem \ref{broadtheorem}  by applying Theorem \ref{largetheorem} with $\mu=
1,$ $\epsilon< 10^{-10}$ and $\alpha= K^{-\epsilon},$ in the same way as Guth shows how Theorem 2.4 follows from Theorem 3.1 in \cite{Gu16}. Keep here Remark \ref{onbroad} in mind, and note that for these choices of $\mu,\epsilon$ and $\alpha,$ we have  $\delta_{trans}\log (2K^{\epsilon}\alpha\mu)\le
10\delta_{trans}\le\epsilon.$

\section{Proof of Theorem  \ref{largetheorem}}
Recall that we  had put $\delta_{trans}:=\epsilon^6,$  so that, if we define $\delta_{deg}:=\epsilon^4$ and $\delta:=\epsilon^2,$ then
$$
\delta_{trans}\ll\delta_{deg}\ll \delta\ll \epsilon< 10^{-10}.
$$
We also set, for given $R\ge 1,$
$$
K=K(\epsilon):=e^{\epsilon^{-10}}\qquad\text{ and }\qquad D=D(\epsilon):=R^{\delta_{deg}}=R^{\epsilon^4}.
$$

\begin{remarks}\label{rem4.1}
a) It is enough to consider the case where $\alpha\mu\le 10^{-5},$ because in the
other case, the exponent $\delta_{trans}\log(K^\epsilon\alpha\mu)$ is very large  and the estimate \eqref{broadest} trivially holds true.
Henceforth, we shall therefore always assume that $\alpha\mu\le 10^{-5}.$

b) It is then also enough to consider the case where $R\ge 1000\, e^{e^{\epsilon^{-12}}}.$
\end{remarks}

\noindent To justify the last claim, notice first that our assumption \eqref{average} implies that $\|f_\tau\|_2\leq 1$. Since there are at most $K(\epsilon)^2$ caps $\tau$, we have $\sum_\tau \|f_\tau\|_2\leq K(\epsilon)^2$. Therefore, we trivially even obtain that when $R\le 1000\, e^{e^{\epsilon^{-12}}},$ then
\begin{eqnarray*}
\int_{B_R}|\E f|^{3.25}&\le& R^3\|f\|_1^{3.25}\le R^3 (\sum\limits_\tau\|f_\tau\|_1)^{3.25}\le  R^3(\sum\limits_\tau\|f_\tau\|_2)^{3.25}\\
&\le& R^3 K(\epsilon)^{2(1/4-2\epsilon)} (\sum\limits_\tau\|f_\tau\|_2)^{3+2\epsilon}
\le R^3 K(\epsilon)^{1/2-4\epsilon}K(\epsilon)^{2(3/2+\epsilon)} (\sum\limits_\tau\|f_\tau\|^2_2)^{3/2+\epsilon}\\
 &\le& C_1(\epsilon) (\sum\limits_\tau\|f_\tau\|^2_2)^{3/2+\epsilon},
\end{eqnarray*}
with $C_1(\epsilon):=(1000\, e^{e^{\epsilon^{-12}}})^3 K(\epsilon)^{7/2-2 \epsilon},$
hence \eqref{broadest}.

\medskip

As usual, we will work with wave packet decompositions of the functions $f$ defined on
${\bf S_\gamma}.$ Following
\cite{Gu16}, we decompose $\Sigma $ into  squares (``caps'') $\theta$ of side length  $R^{-1/2}.$  By
$\omega_\theta$ we shall denote the center of $\theta,$ and by  $\nu(\theta)$  the ``outer'' unit normal  to
${\bf S_\gamma}$ at the point $(\omega_\theta,\phi_\gamma(\omega_\theta))\in {\bf S_\gamma},$ which points into the direction of $(-\nabla \phi_\gamma(\omega_\theta),-1).$
 $\T(\theta)$ will denote a set  of  $R^{1/2}$-separated tubes $T$ of radius $R^{1/2+\delta}$ and length $R,$ which are  all parallel to  $\nu(\theta)$ and for which the corresponding thinner tubes of radius $R^{1/2}$ with the same axes
cover  $B_R.$ We will  write $\nu(T):=\nu(\theta)$ when $T\in\T(\theta).$

Note that  for each $\theta,$ every  point $\xi\in B_R$ lies in
$O(R^{2\delta})$ tubes $T\in\T(\theta).$ We put  $\T:=\bigcup\limits_{\theta} \T(\theta).$ Arguing in the same way as in \cite{Gu16}, Proposition 2.6, we arrive at the following  approximate wave packet decomposition:

\begin{pronr}\label{packets}  Assume that $R$ is sufficiently
large (depending on $\delta$). Then, for any $\gamma\in[-1,1],$ given  $f\in
L^2(\Sigma ),$  we may associate to  each tube $T\in\T$  a function $f_T$ such that the following hold true:
\begin{itemize}
\item[ a)] If $T\in\T(\theta),$ then $\supp f_T\subset 3\theta.$
\item[ b)] If $\xi\in B_R\setminus T,$ then $|\E f_T(\xi)|\le R^{-1000}\|f\|_2.$
\item[c)] For any $x\in B_R,$ we have $|\E f(x)-\sum_{T\in\T}\E f_T(x)|\le
    R^{-1000}\|f\|_2.$
\item[d)] (Essential orthogonality) If $T_1,T_2\in \T(\theta)$ are disjoint, then \hfill \newline $\big|\int f_{T_1} \overline{f_{T_2} }\big | \le R^{-1000} \int_{3\theta} |f|^2.$
\item[e)] $\sum_{T\in\T(\theta)}\int_{\Sigma}|f_T|^2\le C\int_{3\theta}|f|^2.$

\end{itemize}
\end{pronr}
\begin{remark}\label{rem4.3}
Note that since $|\gamma|\le1,$ in this Proposition we have bounds
that are uniform in $\gamma.$ Moreover, note that, the same argument as in Remark \ref{rem4.1} b) shows that, in order to prove  Theorem \ref{largetheorem} it is enough to consider the case where $R$ is sufficiently
large (depending on $\delta,$ i.e., depending on $\epsilon$).
\end{remark}

\smallskip
We next recall the version of the polynomial ham sandwich theorem with non-singular polynomials from \cite{Gu16}.
 If $P$ is a real polynomial on $\R^n,$ we denote by $Z(P):=\{\xi\in\R^n: P(\xi)=0\}$ its  null variety. $P$ is said to be {\it non-singular} if $\nabla P(\xi)\ne 0$ for every point $\xi\in Z(P).$

Then, by Corollary 1.7 in \cite{Gu16} there is a non-zero polynomial $P$ of degree at most $D$ which is a product of non-singular polynomials 
such that the set $\R^3\setminus Z(P)$ is a disjoint union of $\sim D^3$ cells $O_i$ such that,
for every $i,$
\begin{equation}\label{D}
\int_{O_i\cap B_R} (Br_{\alpha} \E f)^{3.25}\sim D^{-3}\int_{B_R} (Br_{\alpha} \E f)^{3.25}.
\end{equation}

We next define $W$ as the $R^{1/2+\delta}$ neighborhood of $Z(P)$ and put $O_i':=(O_i\cap B_R)\setminus W.$

Moreover,  note that if we apply Proposition \ref{packets}  to $f_\tau$ in place of $f$ (what we shall usually do), then by property (a) in Proposition \ref{packets}, for every tube $T\in \T$ the function $f_{\tau,T}$ is supported in an $O(R^{-1/2})$ neighborhood of $\tau.$  Following Guth, we define
$$
\T_i:=\{T\in\T:  T\cap O_i'\ne\emptyset\},\quad f_{\tau,i}:=\sum_{T\in\T_i}
f_{\tau,T},\quad f_{S_\ell,i}:=\sum_{\tau\in\mathcal F_\ell} f_{\tau,i} \quad \text{and}\quad f_i:=\sum_\tau f_{\tau,i}.
$$
Then we can use the following analogue to Lemma 3.2 in \cite{Gu16}:
\begin{lemnr} \label{lemma3.2}
Each tube $T\in\T$ lies in at most $D+1$ of the sets $\T_i.$
\end{lemnr}
We cover $B_R$ with $\sim R^{3\delta}$ balls $B_j$ of radius $R^{1-\delta}.$
Recall Definitions  3.3 and 3.4 from \cite{Gu16}:
\begin{defns} {\rm
a) We define $\T_{j,tang}$ as the set of  all tubes $T\in\T$ that satisfy the following conditions:
$$T\cap W\cap B_j\ne\emptyset,$$
and  if $\xi\in Z(P)$ is any  nonsingular point (i.e., $\nabla P(\xi)\ne 0$) lying in $2B_j\cap 10T,$ then
$$
{\rm angle}(\nu(T), T_\xi Z(P))\le R^{-1/2+2\delta}.
$$
Here, $T_\xi Z(P)$ denotes the tangent space to $Z(P)$ at $\xi,$ and we recall that $\nu(T)$ denotes the unit vector in direction of $T.$
Accordingly, we define
$$ f_{\tau,j,tang}:=\sum_{T\in\T_{j,tang}}f_{\tau,T}\quad\text{and}\quad
f_{j,tang}:=\sum_\tau
f_{\tau,j,tang}.
$$
b) We define $\T_{j,trans}$ as the set of  all tubes $T\in\T$ that satisfy the following conditions:
$$T\cap W\cap B_j\ne\emptyset,
$$
and there exists a nonsingular point $\zeta\in Z(P)$ lying in $2B_j\cap 10T,$ so that
$$
{\rm angle}(\nu(T), T_\zeta Z(P))> R^{-1/2+2\delta}.
$$
Accordingly, we define
$$ f_{\tau,j,trans}:=\sum_{T\in\T_{j,trans}}f_{\tau,T}\quad\text{and}\quad
f_{j,trans}:=\sum_\tau f_{\tau,j,trans}.
$$}
\end{defns}
We also recall Lemmas 3.5 and 3.6 in \cite{Gu16}:
\begin{lemnr}\label{lemma3.5}
Each tube $T\in\T$ belongs to at most ${\rm Poly}(D)=R^{O(\delta_{deg})}$ different sets
$\T_{j,trans}.$
\end{lemnr}
\begin{lemnr}\label{lemma3.6}
For each $j,$ the number  of different $\theta$ so that
$\T_{j,tang}\cap\T(\theta)\ne\emptyset$ is at most $R^{1/2+O(\delta)}.$
\end{lemnr}

Note that the previous lemma makes use of the fact that the Gaussian curvature does not vanish on the surface $\Sigma,$ so that the Gau\ss{} map is a diffeomorphism onto its image.
\medskip

To motivate the next lemma, suppose we have a point $\xi$ contained in a cell $O_i'$. Then it is not hard to see that in the wave packet decomposition of $\E f(\xi)$ essentially only those tubes $T$ should matter which intersect the cell $O_i'$, that is, $T\in\mathbb T_i$. It is thus natural to expect that we may replace $\E f(\xi)$ by $\E f_i(\xi)$ with only a small error.  An analogous statement holds true even for the corresponding broad parts, as the  following analogue to  Lemma 3.7 in \cite{Gu16} shows:

\begin{lemnr}\label{lemma3.7}
If $\xi \in O_i'.$ Then, given  our assumptions on $R$  from Remarks \ref{rem4.1}, we have
$$
Br_\alpha \E f(\xi)\le  Br_{2\alpha} \E f_i(\xi)+R^{-900}\sum_\tau\|f_\tau\|_2.
$$
\end{lemnr}

\begin{proof}
Let  $\xi\in O_i'.$   By Proposition \ref{packets} c), we have
  $$\E f_\tau(\xi)=\sum_{T\in\T}\E f_{\tau,T}(\xi)+O(R^{-1000} \|f_\tau\|_2).
  $$
If $\xi\in T,$ then, $T\cap O_i'\ne\emptyset,$ i.e., $T\in\T_i.$
If $\xi\notin T,$ then Proposition \ref{packets} b) shows that $|\E f_{\tau,T}(\xi)|\le R^{-1000}\|f_\tau\|_2.$ The contribution of
these $T$'s is thus negligible.
\smallskip

 Using the short hand notation ``$\rm neglig$''  for terms which are much smaller than
$R^{-940}\sum_\tau\|f_\tau\|_2$ (and ``${\rm neglig}_\tau$''  for terms  which are much smaller than
$R^{-950}\|f_\tau\|_2$),
we thus have
\begin{equation}\label{approx1}
\E f_\tau(\xi)=\E f_{\tau,i}(\xi)+{\rm neglig}_\tau,
\end{equation}
and summing in $\tau,$
\begin{equation}\label{approx2}
\E f(\xi)=\E f_{i}(\xi)+{\rm neglig}.
\end{equation}
We can assume that $\xi$ is $\alpha$-broad for $\E f$ and that
\begin{equation}\label{approx3}
|\E f(\xi)|\ge R^{-900}\sum_\tau\|f_\tau\|_2.
\end{equation}
 Hence,
 \begin{equation}\label{bigg1}
 |\E f_i(\xi)|\ge|\E f(\xi)|-{\rm neglig}\ge\frac12R^{-900}\sum_\tau\|f_\tau\|_2.
\end{equation}

Now assume that $S_\ell$ is any of the ragged strips used in the definition of $\alpha$-broadness.  Then we have accordingly \begin{equation}\label{approx4}
\E f_{S_\ell}(\xi)=\sum_{\tau\in\mathcal F_\ell} \E f_\tau(\xi)=\sum_{\tau\in\mathcal
F_\ell} \E f_{\tau,i}(\xi)+{\rm neglig}=
\E f_{S_\ell,i}(\xi)+{\rm neglig}.
\end{equation}
Since $\xi$ is $\alpha$-broad for $\E f,$ \eqref{approx4} shows that
$$|\E f_{S_\ell,i}(\xi)|\le |\E f_{S_\ell}(\xi)|+{\rm neglig}\le\alpha|\E f(\xi)|+{\rm neglig}.
$$
Notice also that  by Remarks \ref{rem4.1}, $10^{-5}\ge\alpha\gg K^{-\epsilon}\gg K^{-100}\gg R^{-1}.$ In combination with  \eqref{approx2}, and \eqref{bigg1},  we then obtain that
\begin{equation}\label{broad2}
|\E f_{S_\ell,i}(\xi)|\le \alpha|\E f_i(\xi)|+{\rm neglig}\le 2\alpha |\E f_i(\xi)|
\end{equation}
for every  ragged strip $S_\ell.$
This estimate  shows that $\xi$ is $2\alpha$-broad for $\E f_i,$ and thus the claimed estimate in the lemma follows from
\eqref{approx2} and the assumptions that  we made subsequently.
\end{proof}

Our definition of broadness of points  was chosen  differently from Guth's, since we shall  also need a different notion  of
 ``non-adjacent'' caps. 
 This will be related to the validity of certain bilinear Fourier extension estimates which will be needed in the proof and which will be established later. In order to prepare those,
let us  review some notions and results  concerning such bilinear estimates.
\subsection{Transversality for bilinear estimates}
We shall be brief here and refer  for more details to the corresponding literature dealing
with
bilinear estimates,  for instance \cite{lee05}, \cite{v05}, \cite{lv10}, or \cite{be16}.

Following in particular  and more specifically our discussions in \cite{bmv17}, \cite{bmv18}, we first recall that
 according to Theorem 1.1 in \cite{lee05}, given two open subsets $U_1,U_2\subset [0,1]\times [0,1],$ the proper type of transversality for bilinear estimates is achieved if the
modulus of the following quantity
\begin{align}\label{transs}
\Gamma^\gamma_{z}(z_1,z_2,z_1',z_2'):=	\left\langle
(H\phi_\gamma)^{-1}(z)(\nabla\phi_\gamma(z_2)-\nabla\phi_\gamma(z_1)),\nabla\phi_\gamma(z_2')-\nabla\phi_\gamma(z_1')\right\rangle
\end{align}
is bounded from below for any $z_i=(x_i,y_i),\, z'_i=(x'_i,y'_i)\in U_i\, , i=1,2$, and $z=(x,y)\in U_1\cup
U_2,$ $H\phi_\gamma$
denoting the Hessian of $\phi$.
If such an  inequality holds, then we do have bilinear estimates with constants $C$ that
depend only
on  lower bounds of (the modulus of) in  \eqref{transs}, and on upper bounds for the
derivatives of $\phi_\gamma.$  Note that those upper bounds are independent of $\gamma\in [-1,1];$ we
will be more precise about this later. If $U_1$ and
$U_2$ are sufficiently small (with sizes depending on upper bounds of the first and second
order derivatives of $\phi_\gamma$ and a lower bound for the determinant of $H\phi_\gamma$) this
condition  reduces to the estimate
\begin{equation}\label{Gammalow}
|\Gamma^\gamma_{z}(z_1,z_2)|\geq c>0.
\end{equation}
for $z_i=(x_i,y_i)\in U_i$, $i=1,2$, $z=(x,y)\in U_1\cup U_2$, where
\begin{align}\label{trans}
\Gamma^\gamma_{z}(z_1,z_2):=	\left\langle
(H\phi_\gamma)^{-1}(z)(\nabla\phi_\gamma(z_2)-\nabla\phi_\gamma(z_1)),\nabla \phi_\gamma(z_2)-\nabla\phi_\gamma(z_1)\right\rangle.
\end{align}
The bounds in the corresponding bilinear estimates will then depend on the lower bound $c$ in \eqref{Gammalow}. In contrast to \cite{bmv17}, \cite{bmv18}, where we had to devise quite specific  ``admissible pairs'' of sets $U_1, U_2$ for our bilinear estimates, we shall here only have to consider caps $\tau_1, \tau_2,$ and the required bilinear estimates will be a of somewhat different nature. Nevertheless,  the  geometric transversality conditions that we need here will be the same.

\smallskip

It is easy to check that we explicitly have
\begin{eqnarray} \label{gammaz}
\Gamma^\gamma_{z}(z_1,z_2)
	 &=& 2(y_2-y_1)[x_2-x_1+\gamma(y_1+y_2-y)(y_2-y_1)]\\
	 &=:& 2(y_2-y_1)\, {\bf t}^\gamma_{z}(z_1,z_2). \label{TV1}
\end{eqnarray}

Since $z=(x,y)\in U_1\cup U_2,$ it will be particularly important to look at the
expression
\eqref{TV1} when $z=z_1\in U_1,$ and $z=z_2\in U_2.$  As above, if $U_1$ and
$U_2$ are sufficiently small, we can actually reduce to this case. We then see that  for
our perturbed saddle, still  the difference $y_2-y_1$  in the  $y$-coordinates plays an
important role as for the unperturbed saddle, but  in place  of  the  difference $x_2-x_1$
in the  $x$-coordinates now the quantities
\begin{align}\label{TV2}
	{\bf t}^\gamma_{z_1}(z_1,z_2):=x_2-x_1+\gamma y_2(y_2-y_1)\\
	{\bf t}^\gamma_{z_2}(z_1,z_2):=x_2-x_1+\gamma y_1(y_2-y_1)  \label{TV3}
\end{align}
become relevant.   Observe also that
 \begin{align}\label{symtrans}
 	{\bf t}^\gamma_{z}(z_1,z_2)=-{\bf t}^\gamma_{z}(z_2,z_1).
\end{align}

This definition of transversality motivates the following
\begin{defns}{\rm
a) We say that two caps $\tau_1,\tau_2$ are {\it strongly separated} if 
$$
\min\{|y^c_2-y^c_1|, \max\{|{\bf t}^\gamma_{z^c_1}(z^c_1,z^c_2)|,\,|{\bf t}^\gamma_{z^c_2}(z^c_1,z^c_2)|\}\}\ge 10 \mu^{1/2} K^{-1},\color{black}
$$
where  $z^c_1=(x^c_1,y^c_1)$ denotes the center of $\tau_1$ and $z^c_2=(x^c_2,y^c_2)$ the center of $\tau_2.$
\smallskip

\medskip

b) Following from here again \cite{Gu16}, we define
$$
{\rm Bil}(\E f_{j,tang}):=\sup_{\tau_1,\tau_2\text{ strongly separated}}|\E
f_{\tau_1,j,tang}|^{1/2}|E_\gamma f_{\tau_2,j,tang}|^{1/2}.
$$}
\end{defns}

\begin{remark}\label{Gammasize}
If the caps $\tau_1$ and $\tau_2$ are strongly separated, so that,  say,   $|y^c_2-y^c_1| \ge 10\mu^{1/2} K^{-1}$ and $ |{\bf t}^\gamma_{z^c_2}(z^c_1,z^c_2)|\}\ge 10\mu^{1/2} K^{-1},$ then by \eqref{TV1} we have 
\begin{equation}\label{Gammavar}
|\Gamma^\gamma_{z}(z_1,z_2,z_1',z_2')|\ge 4 \mu K^{-2} \quad \text{for all} \quad z_1,z'_1\in \tau_1,  \, z, z_2,z'_2\in \tau_2.
\end{equation}
\end{remark}

Indeed, one computes that 
\begin{eqnarray*}
\Gamma^\gamma_{z}(z_1,z_2,z_1',z_2')&=&-2\gamma y(y_2-y_1) (y'_2-y'_1)\\
 &+&(y'_2-y'_1)\big(x_2-x_1+\gamma (y_2^2-y_1^2)\big)
+(y_2-y_1)\big(x'_2-x'_1+\gamma ((y'_2)^2-(y'_1)^2)\big)\\
&=& (y'_2-y'_1) {\bf t}^\gamma_{z}(z_1,z_2)+(y_2-y_1){\bf t}^\gamma_{z}(z'_1,z'_2),
\end{eqnarray*}
with ${\bf t}^\gamma_{z}(z_1,z_2)$ defined in \eqref{TV1}.

Now, by \eqref{TV3}, ${\bf t}^\gamma_{z^c_2}(z^c_1,z^c_2)=x^c_2-x^c_1+\gamma (y^c_2+y^c_1- y^c_2)(y^c_2-y^c_1),$ where
$ |{\bf t}^\gamma_{z^c_2}(z^c_1,z^c_2)|\ge 10\mu^{1/2} K^{-1}.$ Since the caps $\tau_1,\tau_2$ have side lengths $\le \mu^{1/2} K^{-1},$ it is easily seen  that 
$|{\bf t}^\gamma_{z_2}(z_1,z_2) -{\bf t}^\gamma_{z^c_2}(z^c_1,z^c_2)|\le 8 \mu^{1/2} K^{-1},$ so  that ${\bf t}^\gamma_{z}(z_1,z_2)$ and ${\bf t}^\gamma_{z^c_2}(z^c_1,z^c_2)$ have the same sign and  $|{\bf t}^\gamma_{z}(z_1,z_2)| \ge 2\mu^{1/2} K^{-1}, $  and  analogously we find  that ${\bf t}^\gamma_{z}(z'_1,z'_2)$ and 
${\bf t}^\gamma_{z^c_2}(z^c_1,z^c_2)$ have the same sign, and that $|{\bf t}^\gamma_{z}(z'_1,z'_2)| \ge 2\mu^{1/2} K^{-1}.$ In a similar way, we see that  $(y'_2-y'_1)$ and  $(y_2-y_1)$ have the same sign as $(y^c_2-y^c_1),$ and that $\min \{|y'_2-y'_1|, |y_2-y_1|\}\ge 2\mu^{1/2} K^{-1},$  since  $|y^c_2-y^c_1|\ge 10 \mu^{1/2} K^{-1}.$ Therefore, 
$
|\Gamma^\gamma_{z}(z_1,z_2,z_1',z_2')|\ge 2(2\mu^{1/2} K^{-1})^2. 
$

\medskip

For any subset $I$ of the family of caps $\tau,$ we define
$$f_{I,j,trans}:= \sum_{\tau\in I} f_{\tau,j,trans}.$$

The remaining part  of this subsection will be  devoted to the proof of the following crucial analogue to  the key Lemma 3.8 in \cite{Gu16}:
\begin{lemnr}\label{lemma3.8}
If $\xi\in B_j\cap W$ and  $\alpha\mu\le 10^{-5},$ then
\begin{equation}\label{crucialbroad}
Br_\alpha\E f(\xi)\le 2\bigg(\sum_IBr_{60\alpha}\E f_{I,j,trans}(\xi)+K^{100} {\rm Bil}(\E
f_{j,tang})(\xi)+R^{-900}\sum_\tau\|f_\tau\|_2\bigg),
\end{equation}
where the first sum is over all possible subsets $I$ of the given family of caps $\tau.$
\end{lemnr}

\begin{remark}
The splitting into a "transversal" and "tangential" part here is as such  not surprising. The crucial point is the presence of the bilinear term. In short, and oversimplified, a given  family of caps $\tau$ will either contain two strongly separated caps, which gives rise to the bilinear term, or  otherwise  we will see by the Geometric Lemma \ref{geometric} that the family cannot contain too many caps, and their contributions can be ``bootstrapped" by means of Lemma \ref{lemma3.5}. For the last point, broadness  will be  of utmost importance (compare \eqref{broad1}).
\end{remark}

\begin{proof}

  Let $\xi\in B_j\cap W. $ We may assume that $\xi$ is $\alpha$-broad for $\E f$ and that  $|\E
f(\xi)|\ge
R^{-900}\sum_\tau\|f\|_2.$ Let
\begin{equation}\label{I}
I:=\{\tau: |\E f_{\tau,j,tang}(\xi)|\le K^{-100} |\E f(\xi)|\}.
\end{equation}
We consider two possible cases:
\smallskip

\noindent{\bf Case 1:}  $I^c$ contains two strongly separated caps $\tau_1$ and
$\tau_2.$ Then trivially
\begin{equation}\label{bili}
|\E f(\xi)|\le K^{100}|\E f_{\tau_1,j,tang}(\xi)|^{1/2}|\E f_{\tau_2,j,tang}(\xi)|^{1/2}\le
K^{100}{\rm Bil}(\E f_{j,tan})(\xi),
\end{equation}
hence \eqref{crucialbroad}.

\noindent{\bf Case 2:}  $I^c$ does not contain two strongly separated caps.
\smallskip

\noindent In this case, we shall make use of the following lemma whose proof will be postponed to Subsection \ref{proofgeom}. Recall the fixed family $\{S_\ell\}_\ell$  of of ragged strips that was associated to our given family of caps $\tau$ (covering $\Sigma$) in Section \ref{reduction}.

\begin{lemnr}{\bf(The Geometric Lemma)}\label{geometric}  Assume that $K\ge 20,$ and let $I^c$ be any  subfamily of the given  family of caps which does not contain two strongly separated caps.

a) If $|\gamma| K^{1/2}>1,$ then all of the caps of $I^c$ belong to the union of  at most 40 of the  families $\mathcal F_m$ associated to long horizontal  ragged strips  $S_m$ of width $\mu^{1/2}K^{-1/4}.$

b) If $|\gamma| K^{1/2}\le1,$ then  either all  of the caps of $I^c$ belong to  the union of at most 3 of  families $\mathcal F_m$ associated to long horizontal ragged strips $S_m$ of width $\mu^{1/2}K^{-1/4}$, or all belong to  the union of at most 40 of the  families $\mathcal F_m$ associated to long vertical ragged  strips $S_m$ of width 
$\mu^{1/2}K^{-1/2}.$
\end{lemnr}
\begin{remark}\label{ongeom}
Note that the two cases in a) and b) basically  match with the corresponding distinction of cases in our definition of $\alpha$-broad points. For our subsequent argument this distinction will, however, not be relevant.
\end{remark}

Using the Geometric Lemma we finish the proof of Lemma \ref{lemma3.8} as follows. We
denote by $\{S_m\}_{m\in\mathcal M}$ the subset of  at most 40 long ragged strips given by the Geometric Lemma. 
By
$$J:=\bigcup\limits_{m\in\mathcal M}\mathcal F_m$$
we denote     the corresponding subset of caps $\tau.$ Then $I^c\subset J,$ i.e.,  $J^c\subset I.$ We write
\begin{align}
	f=\sum_{m\in\mathcal M}\sum_{\tau\in\mathcal F_{m}} f_\tau+\sum_{\tau\in J^c}f_\tau.
\end{align}
Hence,
$$
|\E f(\xi)|\leq\sum_{m\in\mathcal M}|\E f_{S_m}(\xi)|+|\sum_{\tau\in J^c}\E f_\tau(\xi)|.
$$
Since $\xi$ is $\alpha$-broad,
\begin{equation}\label{broad1}
\sum_{m\in\mathcal M} |\E f_{S_m}(\xi)|\le\sum_{m\in\mathcal M} \alpha|\E f(\xi)|\le 40\alpha |\E f(\xi)|\le\frac1{10}|\E f(\xi)|,
\end{equation}
where  the last inequality holds because  we are assuming that $\alpha\le 10^{-5}$ (compare Remark \ref{rem4.1} a)).
Thus,
$$
|\E f(\xi)|\le \frac1{10}|\E f(\xi)|+|\sum_{\tau\in J^c}\E f_\tau(\xi)|,
$$
and therefore
$$
|\E f(\xi)|\le \frac{10}9 |\sum_{\tau\in J^c}\E f_\tau(\xi)|.
$$
Since $\xi\in B_j\cap W,$ by Proposition \ref{packets},
\begin{align}\label{4.21n}
	\E f_\tau(\xi)=\E f_{\tau,j,trans}(\xi)+\E f_{\tau,j,tang}(\xi)+ O(R^{-1000})\|f_\tau\|_2.
\end{align}

Moreover, since $J^c\subset I,$ and  since there are at most $K^2$ caps $\tau,$
\begin{equation}\label{tangesti}
\sum_{\tau\in J^c}|\E f_{\tau,j,tang}(\xi)|\le\sum_{\tau\in I}|\E f_{\tau,j,tang}(\xi)|\le
K^{-100}\sum_{\tau\in I}|\E f(\xi)|\le K^{-98}|\E f(\xi)|,
\end{equation}
where the second inequality is a consequence of the definition of $I.$
Thus,
\begin{eqnarray*}
\frac 9{10}|\E f(\xi)|&\le& |\sum_{\tau\in J^c}\E f_{\tau,j,trans}(\xi)|+K^{-98}|\E
f(\xi)|+\sum_\tau R^{-1000}\|f_\tau\|_2\\
&=&|\E f_{J^c,j,trans}(\xi)|+K^{-98}|\E f(\xi)|+\sum_\tau R^{-1000}\|f_\tau\|_2,
\end{eqnarray*}
and hence, since $|\E f(\xi)|\ge R^{-900}\sum\|f_\tau\|_2,$
\begin{equation}\label{11-9}
|\E f(\xi)|\le\frac{11}9|\E f_{J^c,j,trans}(\xi)|.
\end{equation}
\medskip
It will then finally suffice to show  that $\xi$ is $60\alpha$-broad for $\E g,$ where  $g:= f_{J^c,j,trans}.$
To this end let us set $g_\tau:=f_{\tau,j,trans},$ if $\tau\in J^c,$ and zero otherwise,  so that
$$g=\sum g_\tau.$$

Observe first that  by \eqref{4.21n}
$$
|\E f_{\tau,j,trans}(\xi)|\le|\E f_\tau(\xi)|+|\E f_{\tau,j,tang}(\xi)|+{\rm neglig},
$$
so that  if $\tau\in J^c\subset I,$ then by the definition of $I,$
\begin{equation}\label{transesti}
|\E f_{\tau,j,trans}(\xi)|\le|\E f_\tau(\xi)|+K^{-100}|\E f(\xi)|+{\rm neglig}.
\end{equation}

\medskip

We have to show that
$$
|\E g_{S_\ell}(\xi)|\le  60\alpha |\E g(\xi)|
$$
for all ragged strips $S_\ell.$  But,
$g_{S_\ell}=\sum_{\tau\in \mathcal F_\ell\cap J^c}f_{\tau,j,trans},$ and therefore the following two cases can arise:
\begin{enumerate}
	\item If $\ell\in \mathcal M$, then $\F_\ell\cap J^c=\emptyset$.
	\item If $\ell\notin \mathcal M$,  then by our construction of  the set $J$ there is a collection $\{\mathcal F_r\}_{r\in\mathcal R}$ of at most 40 families (possibly empty)  associated to short vertical ragged strips so that $\mathcal F_\ell\cap J=\cup_{\mathcal R} \mathcal F_r$ (cf. Figure \ref{grid}).
\end{enumerate}

\begin{figure}[h]\begin{center}
  \includegraphics[scale=0.25]{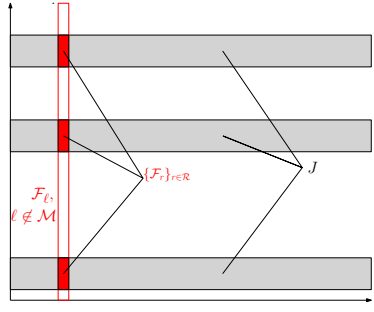}
\caption{\small Intersection of long horizontal and long vertical strips}
\label{grid}
\end{center}\end{figure}

\smallskip

Observe first that by summing \eqref{4.21n} over all $\tau\in\mathcal F_\ell\cap J^c$ we obtain
$$
	|\E g_{S_\ell}(\xi)|\leq|\sum_{\tau\in \mathcal F_\ell\cap J^c}\E f_{\tau}(\xi)|+\sum_{\tau\in \mathcal F_\ell\cap J^c}|\E f_{\tau,j,tang}(\xi)|+{\rm neglig}.
$$
By \eqref{tangesti}, the second term can again be estimated  by
$$
	\sum_{\tau\in \mathcal F_\ell\cap J^c}|\E f_{\tau,j,tang}(\xi)|
	\leq K^{-98}|\E f(\xi)|.
$$
Case (i) is trivial.  In case (ii), we write
$$
\sum_{\tau\in \mathcal F_\ell\cap J^c}\E f_{\tau}(\xi)=\E f_{S_\ell}(\xi)-\sum_{r\in\mathcal R} \E f_{S_r}(\xi).
$$
Since $\xi$ is $\alpha$-broad for $\E f$,
 both terms are estimated using again broadness:
$$
|\sum_{\tau\in \mathcal F_\ell\cap J^c}\E f_{\tau}(\xi)|\leq 41\alpha |\E f(\xi)|.
$$
Since $\alpha\ge K^{-\epsilon}\gg 10 K^{-98},$ in combination with
\eqref{11-9} we conclude that
\begin{eqnarray*}
|\E g_{S_\ell}(\xi)| &\le &41\alpha |\E f(\xi)|+K^{-98}|\E f(\xi)|+{\rm neglig}\le(41+1/2) \alpha |\E f(\xi)| \\
&\le&  60 \alpha |\E f_{J^c,j,trans}(\xi)| =60 \alpha |\E g(\xi)|.
\end{eqnarray*}

This completes the proof of Lemma \ref{lemma3.8}.
\end{proof}
\medskip

The contribution by the bilinear term in \eqref{crucialbroad} will be controlled by means of the following analogue to Proposition 3.9 in \cite{Gu16}:
\begin{pronr}\label{prop3.9} We have
$$
\int_{B_j\cap W}{\rm Bil}(\E f_{j,tang})^{3.25}\le C_\epsilon R^{O(\delta)+\epsilon/2}\bigg(\sum_\tau\int
|f_\tau|^2\bigg)^{3/2+\epsilon}.
$$
\end{pronr}

With Proposition  \ref{prop3.9} at hand, the rest  of the proof of Theorem \ref{largetheorem}, which we shall detail in the next subsection,  will be a  literal copy of the arguments in pages 396-398 of \cite{Gu16}.
\smallskip

The proof of this  proposition can easily be reduced to the following analogue to Lemma 3.10 in \cite{Gu16}. We shall give some details below. It is in this lemma where we shall need the full thrust of the strong separation condition between caps $\tau_1$ and $\tau_2.$   Suppose we have covered $B_j\cap W$ with a minimal number of cubes $Q$  of side length $R^{1/2},$   and denote by $\T_{j,tang,Q}$ the set of all tubes $T$ in $\T_{j,tang}$ such that $10T$ intersects $Q.$
\begin{lemnr}\label{lemma4.11} Fix $j,$ i.e., a ball $B_j.$
If $\tau_1,\tau_2$ are strongly separated cups, then for any of the cubes $Q$  we have
\begin{eqnarray*}
&&\int_Q|\E f_{\tau_1,j,tang}|^2|\E f_{\tau_2,j,tang}|^2\\
&&\phantom\qquad\le
R^{O(\delta)}R^{-1/2}
\big(\sum_{T_1\in\T_{j,tang,Q}}\|f_{\tau_1,T_1}\|_2^2\big)
\big(\sum_{T_2\in\T_{j,tang,Q}}
\|f_{\tau_2,T_2}\|_2^2\big)+{\rm neglig}.
\end{eqnarray*}
\end{lemnr}

Indeed, the main  ingredient in Guth's argument that needs to be checked here is the following geometric property (compare p. 402 in \cite{Gu16}):

\begin{lemnr}\label{anglesep}
If $\tau_1$ and $\tau_2$ are two strongly separated caps, then, for any two points $z_1=(x_1,y_1)\in \tau_1$ and $z_2=(x_2,y_2)\in \tau_2$ the  angle between the normals to ${\bf S_\gamma}$ at the corresponding points on ${\bf S_\gamma}$ is $\gtrsim K^{-1}.$
\end{lemnr}
\begin{proof}
By $N_\gamma(x,y)$ we denote the following normal to our surface ${\bf S_\gamma}$  at $(x,y,\phi_\gamma(x,y))\in {\bf S_\gamma}:$
\begin{equation}\label{Normal}
N_\gamma(x,y):=
\left(
\begin{array}{c}
  {\,}^t\nabla\phi_\gamma(x,y)   \\
-1
\end{array}
\right).
\end{equation}
Note that these normal vectors are of size $|N_\gamma(x,y)|\sim 1.$
Since $\nabla\phi_\gamma(x,y)=(y, x+\gamma y^2),$ we see that
\begin{equation}\label{deltanabla}
|\nabla\phi_\gamma(x_2,y_2)-\nabla\phi_\gamma(x_1,y_1)|\ge |y_2-y_1| \gtrsim 10K^{-1}
\end{equation}
since $\tau_1$ and $\tau_2$ are strongly separated. This implies the claim about the angle.
\end{proof}

With this at hand, we can follow Guth to deduce from Lemma \ref{lemma4.11}  the following $L^4$ estimate
\begin{equation}\label{L4}
\|{\rm Bil}(\E f_{j,tang})\|_{L^4(B_j\cap W)}\le R^{O(\delta)}R^{-1/8}(\sum_\tau\|f_{\tau,j,tan}\|_2^2)^{1/2}+{\rm neglig},
\end{equation}
which corresponds to inequality (43) in \cite{Gu16}.
 Indeed,  we can use the standard estimate
$$
\|\E f\|_{L^2(B_R)}\lesssim R^{1/2}\|f\|_2
$$
to deduce that
$$
\|{\rm Bil}(\E f_{j,tang})\|_{L^2(B_j\cap W)}\le
R^{1/2}(\sum_\tau\|f_{\tau,j,tan}\|_2^2)^{1/2}.
$$
From this and \eqref{L4}, by H\"older's inequality we get  for $2\le p\le 4,$
$$
\int_{B_j\cap W}{\rm Bil}(\E f_{j,tang})^{p}\lesssim R^{O(\delta)} R^{\frac52-\frac34p}\bigg(\sum_\tau\|f_{\tau,j,tang}\|_2^2
\bigg)^{p/2}+{\rm neglig}.
$$
Lemma \ref{lemma3.6} tells us that $\mathbb T_{j,tang}$ contains tubes in only $R^{O(\delta)}R^{1/2}$ directions. Hence, each function $f_{\tau,j,tang}$ is supported on at most $R^{O(\delta)}R^{1/2}$ caps $\theta.$ By Proposition \ref{packets},
$$
\oint_{\theta}|f_{\tau,j,tang}|^2\lesssim \oint_{10\theta}|f_{\tau}|^2\lesssim 1.
$$
Adding the contribution of  $R^{O(\delta)}R^{1/2}$ caps $\theta,$ we get $\int |f_{\tau,j,tang}|^2\le C R^{O(\delta)}R^{-1/2}.$
Since there are $K^2\ll R^{O(\delta)}$ caps $\tau,$ this implies that $\sum_\tau\|f_{\tau,j,tang}\|_2^2\le C R^{O(\delta)}R^{-1/2}.$
Hence, we get, for $p>3,$ $\epsilon<2(p-3),$
$$
\int_{B_j\cap W}{\rm Bil}(\E f_{j,tang})^{p}\lesssim R^{O(\delta)} R^{\frac52-\frac34p-\frac12(\frac p2-\frac32-\epsilon)}\bigg(\sum_\tau\|f_{\tau,j,tang}\|_2^2 \bigg)^{3/2+\epsilon}.
$$
This finishes the proof of Proposition \ref{prop3.9}, for $p=3.25=\frac{13}4.$

\bigskip

\noindent{\it Proof of  Lemma \ref{lemma4.11}.}   Let  $\tau_1$ and $\tau_2$ be two strongly separated caps, and assume without loss of generality that $\min\{|y^c_2-y^c_1|, |{\bf t}^\gamma_{z^c_2}(z^c_1,z^c_2)|\}\ge 10\mu^{1/2} K^{-1},$ where $z_1^c=(x^c_1,y^c_1)$ denotes the center of $\tau_1$ and $z_2^c=(x^c_2,y^c_2)$  the center of $\tau_2.$ 

Following in a first step a standard argument as in  \cite{Gu16} based on Plancherel's theorem,  and making use of Proposition \ref{packets} we see that
\begin{eqnarray}\nonumber
&&\int_Q|\E f_{\tau_1,j,tang}|^2|\E f_{\tau_2,j,tang}|^2\\
&\le&\sum_{T_1,T_1',T_2,T_2'\in\T_{j,tang,Q}}\int  \,\E f_{\tau_1,T_1}
\E f_{\tau_2,T_2}\overline{\E f_{\tau_1',T_1'}\E f_{\tau_2',T_2'}}+{\rm neglig}\nonumber \\
&=&\sum_{T_1,T_1',T_2,T_2'\in\T_{j,tang,Q}} (f_{\tau_1,T_1}\,d\sigma_\gamma
*f_{\tau_2,T_2}\,d\sigma_\gamma) \overline{(f_{\tau_1',T_1'}\,d\sigma_\gamma *f_{\tau_2',T_2'}\,d\sigma_\gamma)}+{\rm neglig}.\label{quadrup}
\end{eqnarray}
Here $\sigma_\gamma$ denotes the surface carried measure on ${\bf S_\gamma}$  chosen so that $\E f=\widehat{\tilde f d\sigma_\gamma},$ if  we set $\tilde f(z,\phi_\gamma(z)):=f(z).$

\smallskip

For each tube $T,$ we denote by $\theta(T)$ the cap $\theta$ so that $T\in\T(\theta),$ and let $\omega(T)$ be the center of $\theta(T).$  By $\tilde \omega(T):=(\omega(T),\phi_\gamma(\omega(T))$ we denote the corresponding point on ${\bf S_\gamma}.$

A given term in the first sum is not negligible only if there are four points
$z_1,\,z_1'\in \tau_1,$ $z_2,\,z_2'\in\tau_2$ that satisfy
\begin{equation}\label{sumcond1}
(z_1,\phi_\gamma(z_1))+(z_2,\phi_\gamma(z_2))=(z_1',\phi_\gamma(z_1'))+(z_2',\phi_\gamma(z_2'))
\end{equation}
and
\begin{equation}\label{sumcond2}
(z_i,\phi_\gamma(z_i))=\tilde \omega(T_i)+O(R^{-1/2+\delta}),\quad
(z_i',\phi_\gamma(z_i'))=\tilde \omega(T_i')+O(R^{-1/2+\delta}),\qquad i=1,2.
\end{equation}

Let us denote by ${\bf S}_i$ the   piece of the surface ${\bf S_\gamma}$ corresponding to $\tau_i, \, i=1,2$  (which are ``genuine'' caps). Since the caps $\tau_1$ and $\tau_2$ are strongly separated, by Lemma \ref{anglesep} these two  subsurfaces are transversal, so that we  can locally define the {\it intersection curve}
$$
\Pi_{z_1,z_2'}:=[{\bf S}_1+(z_2',\phi(z_2'))]\cap[{\bf S}_2+(z_1,\phi(z_1))].
$$
Note that by \eqref{sumcond1}
$$
(z_1,\phi_\gamma(z_1))+(z_2,\phi_\gamma(z_2))=(z_1',\phi_\gamma(z_1'))+
(z_2',\phi_\gamma(z_2'))\in \Pi_{z_1,z_2'}.
$$
Set $\psi(z):= \phi_\gamma(z-z_1)+\phi_\gamma(z_1)-\phi_\gamma(z-z_2')-\phi_\gamma(z_2').$ Then, the orthogonal
projection of the curve $\Pi_{z_1,z_2'}$ on the $z$ - plane  is the  curve  given by
$\{z:\;\psi(z)=0\}$ (just consider $z:=z_1+z_2=z'_1+z'_2$ for $z$ when \eqref{sumcond1} is satisfied).

We introduce a parametrization by arc length $z(t),\, t\in  J,$  of this curve, where
$t$ is from an  open interval $ J.$
Notice that this curve $z(t)$ depends on the choices of  the points $z_1$ and $z_2'.$
By
$$
z'_1(t):=z(t)-z'_2\quad \text{ and } \quad z_2(t):=z(t)-z_1
$$
we denote the corresponding curves on ${\bf S}_1$ and ${\bf S}_2,$ respectively.
We may assume that $0\in J$ and $z'_1(0)=z'_1, z_2(0)=z_2.$ Then, for $z_1,z'_2$ fixed, the pairs $(z'_1(t), z_2(t)), t\in J,$ locally provide all solutions
$(z'_1,z_2)$ to \eqref{sumcond1}.

Note that  $\psi(z(t))\equiv 0$ implies that

\begin{equation}\label{purp1}
\langle \nabla\phi_\gamma(z_2(t))-\nabla\phi_\gamma(z'_1(t)), \frac {dz}{dt}(t) \rangle=0 \quad \text{for every} \quad t\in J.
\end{equation}
Note also that $(z'_1,z_2):=(z_1,z'_2)$ is a solution of  \eqref{sumcond1}, so that we may assume that there is some $t_2\in J$ such that $z'_2=z_2(t_2).$ Recall also that $z_2(0)=z_2.$
\medskip

Recall the normal $N_\gamma(x,y)$  to the  surface ${\bf S_\gamma}$  at the point  $(x,y,\phi_\gamma(x,y))\in {\bf S_\gamma}$ from \eqref{Normal},
and note that  the angle between the tube $T_i$ and $N_\gamma(z_i),\, i=1,2,$ is bounded by $R^{-1/2}.$
\smallskip

Since $T_1,T_2,T_1',T_2'$ lie in $\T_{j,tang,Q},$ we then obviously have
\begin{eqnarray}\nonumber
R^{-1/2+2\delta}&\ge &|\det(N_\gamma(z_1),N_\gamma(z_2),N_\gamma(z_2'))|=|\det(N_\gamma(z_1),N_\gamma(z_2),N_\gamma(z_2')-N_\gamma(z_2))|\\
&=&\Big|\int_{0}^{t_2}
\det(N(z_1),N(z_2),\frac{d N_\gamma(z_2(t))}{dt})\,dt\Big|. \label{estdet}
\end{eqnarray}
 For a given $t,$
\begin{eqnarray*}
&&\hskip-0.5cm\det\Big(N_\gamma(z_1),N_\gamma(z_2),\frac{d N_\gamma(z_2(t))}{dt}\Big)=\det \left(\begin{array}{ccc}
      {\,}^t\nabla\phi_\gamma( z_1) & {\,}^t\nabla\phi_\gamma( z_2)& H\phi_\gamma( z_2(t))\cdot \trans (\frac {dz}{dt}(t))  \\
    -1 & -1 & 0
\end{array}\right)\\
&=&\det \left({\,}^t\nabla\phi_\gamma( z_1)-{\,}^t\nabla\phi_\gamma( z_2), H\phi_\gamma( z_2(t))\cdot \trans (\frac {dz}{dt}(t)) \right)\\
&=&\det H\phi_\gamma( z_2(t))\, \det \left(H\phi_\gamma( z_2(t))^{-1}\cdot  \left(
{\,}^t\nabla\phi_\gamma( z_1)-{\,}^t\nabla\phi_\gamma( z_2)\right), \trans (\frac {dz}{dt}(t)) \right).
\end{eqnarray*}

Since $|\frac {dz}{dt}(t))|=1$ and $\det H\phi_\gamma( z(t))=1,$ in combination with \eqref{purp1} we thus see that
\begin{eqnarray*}
&&\hskip-1cm\Big|\det\big(N_\gamma(z_1),N_\gamma(z_2),\frac{d N_\gamma(z_2(t))}{dt}\big)\Big|\\
&=&\frac{\big|\langle H\phi_\gamma ( z_2(t))^{-1}\left(
{\,}^t\nabla\phi( z_1)-{\,}^t\nabla\phi(
z_2)\right), \nabla\phi_\gamma(z_2(t))-\nabla\phi_\gamma(z'_1(t))\rangle\big|}
{| \nabla\phi_\gamma(z_2(t))-\nabla\phi_\gamma(z'_1(t))|}\\
&=&\frac{\big| \Gamma^\gamma_{z_2(t)}(z_1,z_2,z'_1(t), z_2(t))\big|}{ |\nabla\phi_\gamma(z_2(t))-\nabla\phi_\gamma(z'_1(t))|}.
\end{eqnarray*}

Note that here $ |\nabla\phi_\gamma(z_2(t))-\nabla\phi_\gamma(z'_1(t))|\le 4.$ Moreover, by our assumptions and Remark \ref{Gammasize}, we have 
$|\Gamma^\gamma_{z_2(t)}(z_1,z_2,z'_1(t),z_2(t))|\ge 4 \mu K^{-2}.$ 

Therefore
$$
\Big|\det\big(N_\gamma(z_1),N_\gamma(z_2),\frac{d N_\gamma(z_2(t))}{dt}\big)\Big|\ge \mu K^{-2},
$$
and   since the integrand in  \eqref{estdet} has constant sign, we see that
$$
R^{-1/2+2\delta}\ge \big|\int_{0}^{t_2}\mu K^{-2}\, dt\big|.
$$
Hence,
$|t_2|\le K^2 R^{-1/2+2\delta}$ and, since the curve $t\mapsto z_2(t)$ is parametrized by arc length, we find that
$|z_2-z_2'|\lesssim  K^2 R^{-1/2+2\delta}.$ Since  $z_1-z'_1=z'_2-z_2$ by \eqref{sumcond1}, we also get  $|z_1-z_1'|\lesssim  K^2 R^{-1/2+2\delta}.$ 

In a similar way, we see that  $|z_1-z_1'|\le  K^2 R^{-1/2+2\delta}.$

 Hence, given $T_1$ and $T_2,$ there are at most $R^{O(\delta)}$ possible tubes $T_1',T_2'$  which give a non-negligible contribution to \eqref{quadrup}, and
by Schur's lemma this implies that
$$
\int_Q|\E f_{\tau_1,j,tang}|^2|\E f_{\tau_2,j,tang}|^2
\le  R^{O(\delta)} \sum_{T_1,T_2\in\T_{j,tang,Q}}\int |f_{\tau_1,T_1}\,d\sigma_\gamma * f_{\tau_2,T_2}\,d\sigma_\gamma|^2
+{\rm neglig}.
$$

Finally, note that Lemma \ref{anglesep} implies that  $T_1\cap T_2$ is contained in a cube of side length $KR^{1/2+\delta}.$   Hence, the same reasoning used to prove inequality (38) in \cite{Gu16} leads to
$$
\int |f_{\tau_1,T_1}\,d\sigma_\gamma * f_{\tau_2,T_2}\,d\sigma_\gamma|^2\le R^{-1/2}\|f_{\tau_1,T_1}\|_2^2\, \|f_{\tau_2,T_2}\|_2^2,
$$
and combining these two estimates we complete the   proof of Lemma \ref{lemma4.11}.

\qed

\color{black}
\subsection{Completing the  proof of Theorem \ref{largetheorem}}

Following \cite{Gu16}, pp. 396--398, we use induction on the size of  $R,$ the radius of $B_R.$  Moreover, for  given $R,$ we also induct on the size of $\sum_\tau\int |f_\tau|^2.$
Here we understand that a positive quantity is {\it of size} $2^k,\, k\in\Z,$ if it lies in the interval $(2^{k-1}, 2^{k}].$
\smallskip

\noindent\bf Bases of induction. a) \rm
We recall from Remark \ref{rem4.1} b) that for $1\le R\le1000 \, e^{e^{\epsilon^{-12}}}$
$$
\int_{B_R}|\E f|^{3.25}\le C_4(\epsilon) (\sum\limits_\tau\|f_\tau\|^2_2)^{3/2+\epsilon},
$$
so that the estimate in Theorem \ref{largetheorem} holds true for this range of $R$'s.
\medskip

{\bf b)} Also, if  $\sum_\tau\int |f_\tau|^2\le R^{-1000},$ then the estimate in Theorem \ref{largetheorem} holds trivially, since
$$
\int_{B_R}|\E f|^{3.25}\le R^3\|f\|_1^{3.25}\le R^3\|f\|_2^{3.25}\le
R^{-100}\|f\|_2^{3+2\epsilon}\le K(\epsilon)^2
R^{-100}\big(\sum_\tau\int|f_\tau|^2\big)^{3/2+\epsilon}.
$$
In the induction procedure, it will thus suffice to show that in each step we can reduce to  situations where either $R,$ or
$\sum_\tau\int |f_\tau|^2,$ becomes smaller  by a factor  $\le 1/2,$  until we go below one of the thresholds described in a), or b).

\medskip
We shall show that inequality \eqref{broadest} of Theorem \ref{largetheorem} will then hold with the constant   $C_\epsilon:=\max\{K(\epsilon)^2, C_4(\epsilon)\}.$

\medskip

\noindent\bf Induction hypotheses. \rm Assume that Theorem \ref{largetheorem} holds for all radii $\le R/2,$ or, given $R,$ for all functions $g$ in place of $f$ such that
$\sum_\tau\int|g_\tau|^2\le\frac12\sum_\tau\int|f_\tau|^2$ and every $\mu\ge 1.$

\medskip

Write
\begin{equation}\label{cell-wall}
\int_{B_R} (Br_\alpha\E f)^{3.25}=\sum_i\int_{B_R\cap O_i'} (Br_\alpha\E
f)^{3.25}+\int_{B_R\cap W} (Br_\alpha\E f)^{3.25}.
\end{equation}

\noindent{\bf  Case 1.  Assume that the first term (cellular term) dominates \eqref{cell-wall}.}
In this case, by \eqref{D} there will be $\sim D^3$ cells $O_i',$ and for each of them
$$
\int_{B_R\cap O_i'} Br_\alpha\E f^{3.25}\sim D^{-3}\int_{B_R} Br_\alpha\E f^{3.25}.
$$
In combination with  Lemma \ref{lemma3.7}, then, for every $i,$
\begin{eqnarray}\nonumber
\int_{B_R} (Br_\alpha\E f)^{3.25}&\sim& D^3 \int_{B_R\cap O_i'} (Br_\alpha\E f)^{3.25}\\
&\lesssim& D^3
\int_{B_R\cap O_i'}
(Br_{2\alpha}\E f_i)^{3.25}+R^{-900}(\sum_\tau\|f_\tau\|_2)^{3.25}. \label{D3}
\end{eqnarray}
If the second term in \eqref{D3} dominates, then, since $\alpha\ge K^\epsilon,$
$R^{\delta_{trans}\log(K^\epsilon\alpha\mu)}\ge R^{\delta_{trans}}\ge 1,$ and that finishes
the proof.

If the first term in \eqref{D3} dominates, we use Lemma \ref{lemma3.2} and the following immediate analogue to Lemma 2.7 in \cite{Gu16}:
\begin{lemnr}
Consider some subsets $\T_i\subset \T$ indexed by $i\in \mathcal I.$ If each tube $T$ belongs to at
most $\kappa$ of the
subsets $\{\T_i\}_{i\in \mathcal I},$ then, for every $\theta,$
$$
\sum_{i\in
\mathcal I}\int_{3\theta}|f_{\tau,i}|^2\le\kappa\int_{10\theta}|f_\tau|^2,
$$
and
$$
\sum_{i\in \mathcal I}\int|f_{\tau,i}|^2\le\kappa\int |f_\tau|^2.
$$
\end{lemnr}
Applying this lemma in combination with Lemma \ref{lemma3.2},  we see that for each $\tau,$
$$
\sum_{i}\int|f_{\tau,i}|^2\le (D+1) \int|f_\tau|^2,
$$
and therefore
$$
\sum_{i}\sum_\tau\int|f_{\tau,i}|^2\le (D+1) \sum_\tau\int|f_\tau|^2.
$$
Now, recall that there are $\sim D^{3}$ indices $i.$  Thus we can choose and fix an index $i_0$ such
that
\begin{equation}\label{1/2}
\sum_\tau\int|f_{\tau,i_0}|^2\lesssim  D\,D^{-3} \sum_\tau\int|f_\tau|^2=D^{-2}
\sum_\tau\int|f_\tau|^2\ll \frac12\sum_\tau\int|f_\tau|^2.
\end{equation}
We finish this case by applying  the induction hypothesis (on the size of
$\sum_\tau\int|g_\tau|^2$) to the function
$f_{i_0}:=\sum_\tau f_{\tau,i_0}.$
Note that the support of $f_{\tau,i_0}$ is a tiny neighborhood of $\tau.$ For this reason we
need $\mu$ in the
statement of Theorem \ref{largetheorem}, so that here we can  apply the induction hypothesis with $2\mu$
in place of $\mu.$

To this end, note also that
$$
\oint_{B(\omega,R^{-1/2})} |f_{\tau,i_0}|^2\le C\oint _{B(\omega,10R^{-1/2})}
|f_{\tau}|^2\le C,
$$
where the first inequality is a consequence of the following   immediate analogue to Lemma 2.8 in \cite{Gu16}:
\begin{lemnr}\label{lemma2.8}
If $\T_i\subset \T,$ then for any cap $\theta,$ and any $\tau,$
$$
\int_{3\theta}|f_{\tau,i}|^2\le\int_{10\theta}|f_\tau|^2.
$$
\end{lemnr}
We then apply our induction hypothesis to $\frac1{\sqrt{2C}}f_{i_0}=\frac1{\sqrt{2C}}\sum_\tau f_{\tau,i_0}.$ Since we assume that the first term in  \eqref{D3} dominates, this yields
\begin{eqnarray*}
\int_{B_R} (Br_\alpha\E f)^{3.25}
&\lesssim&D^3 \int_{B_R\cap O'_{i_0}} (Br_{2\alpha}\E f_{i_0})^{3.25} \\
&\le&  (2C)^{1/8-\epsilon}D^3C_\epsilon R^\epsilon
\bigg(\sum_\tau\int|f_{\tau,{i_0}}|^2\bigg)^{3/2+\epsilon}R^{\delta_{trans}\log(K^\epsilon2\alpha2\mu)},
\end{eqnarray*}
and  thus by \eqref{1/2}
\begin{eqnarray*}
\int_{B_R} (Br_\alpha\E f)^{3.25}&\le& C_1 D^3C_\epsilon R^\epsilon
\bigg(D^{-2}\sum_\tau\int|f_{\tau}|^2\bigg)^{3/2+\epsilon}R^{\delta_{trans}\log(K^\epsilon\alpha\mu)}
R^{c\delta_{trans}}\\
&=&C_1 C_\epsilon R^\epsilon
\bigg(\sum_\tau\int|f_{\tau}|^2\bigg)^{3/2+\epsilon}R^{\delta_{trans}\log(K^\epsilon\alpha\mu)}
D^{-2\epsilon}R^{c\delta_{trans}}\\
&\le& C_\epsilon R^\epsilon
\bigg(\sum_\tau\int|f_{\tau}|^2\bigg)^{3/2+\epsilon}R^{\delta_{trans}\log(K^\epsilon\alpha\mu)},
\end{eqnarray*}
closing the induction.

\medskip

\noindent{\bf  Case 2. Assume that the  second term (wall term) dominates \eqref{cell-wall}.}
In this case we apply Lemma \ref{lemma3.8} to obtain
\begin{eqnarray}\nonumber
\int_{B_R} (Br_\alpha\E f)^{3.25}&\le& C_\epsilon \sum_{j}\int_{B_j\cap W}\sum_ I (Br_{150\alpha}\E
f_{I,j,trans})^{3.25}\\
&+& C K^{325}\sum_j\int_{B_j\cap W} {\rm Bil}(\E f_{j,tang})^{3.25}+C \left(R^{-900}\sum_\tau\|f_\tau\|_2\right)^{3.25}  \label{three}
\end{eqnarray}
(note that the number of all possible subsets I of the given family of caps is only a constant depending on $\epsilon$).

Again, if the third term of this last sum dominates, the proof is easily finished.
\smallskip

If the second term dominates, then by Proposition \ref{prop3.9}, since $K\ll R,$
\begin{eqnarray*}
\int_{B_R} Br_\alpha\E f^{3.25}&\lesssim & C_\epsilon K^{325}\sum_j\int_{B_j\cap W} {\rm Bil}(\E f_{j,tang})^{3.25}\\
&\le& C_\epsilon K^{325}R^{O(\delta)+\epsilon/2}\bigg(\sum_\tau\int|f_\tau|^2\bigg)^{3/2+\epsilon} \le
CR^{\epsilon}\bigg(\sum_\tau\int|f_\tau|^2\bigg)^{3/2+\epsilon}.
\end{eqnarray*}

This finishes the proof in this case.
\smallskip

Finally, assume that  the first term in \eqref{three} dominates.  Then, since the ball $B_j$ has radius
$R^{1-\delta}<\frac R2,$ we shall induct on the size of $R.$
Note also that $f_{\tau,j,trans,I}$ is supported in a tiny neighborhood of $\tau,$ so we shall again   apply the induction hypothesis with $2\mu$
in place of $\mu.$

 By Lemma \ref{lemma2.8},
$$
\oint_{B(\omega, R^{-1/2})}|f_{I,j,trans,\tau}|^2\le\oint_{B(\omega,R^{-1/2})}|f_\tau|^2\le C,
$$
which implies the same kind of control over larger balls of radius $(R^{1-\delta})^{-1/2}.$
Thus, $\frac1C f_{I,j,trans}$ satisfies the induction hypothesis of Theorem
\ref{largetheorem},  and therefore
$$
\int_{B_j\cap W}(Br_{150\alpha}\E f_{I,j,trans})^{3.25} \leq C_\epsilon R^{\epsilon(1-\delta)}
\bigg(\sum_{\tau\in I}\int|f_{\tau,j,trans}|^2\bigg)^{3/2+\epsilon}R^{\delta_{trans}(1-\delta)\log(4K^\epsilon\alpha\mu)}.
$$
By Lemma \ref{lemma3.5},
$$
\sum_{j}\int|f_{\tau,j,trans}|^2\le {\rm Poly}(D)\int |f_\tau|^2.
$$
Moreover,
$$
\sum_{I,j}\bigg(\sum_{\tau\in I} \int|f_{\tau,j,trans}|^2\bigg)^{3/2+\epsilon}
\le \sum_I\bigg(\sum_j\sum_{\tau\in I} \int|f_{\tau,j,trans}|^2\bigg)^{3/2+\epsilon}.
$$
Since  there are at most $M_\epsilon$ families $I,$ combining these estimates we see that
$$
\sum_{I,j}\bigg(\sum_{\tau\in I} \int|f_{\tau,j,trans}|^2\bigg)^{3/2+\epsilon}\le M_\epsilon {\rm Poly}(D)\bigg(\sum_\tau\int|f_\tau|^2\bigg)^{3/2+\epsilon},
$$
and thus finally
\begin{eqnarray*}
\int_{B_R}(Br_{2\alpha}\E f)^{3.25} &\le&M_\epsilon C_\epsilon {\rm Poly}(D) R^{\epsilon(1-\delta)}
\bigg(\sum_\tau\int |f_\tau|^2\bigg)^{3/2+\epsilon}R^{\delta_{trans}(1-\delta)\log(4K^\epsilon\alpha\mu)}\\
&\le&(M_\epsilon {\rm Poly}(D) R^{-\delta\epsilon+c\delta_{trans}})\,  C_\epsilon R^{\epsilon}
\bigg(\sum_\tau\int|f_\tau|^2\bigg)^{3/2+\epsilon}
R^{\delta_{trans}\log(K^\epsilon\alpha\mu)}.
\end{eqnarray*}
 By our choices of $\delta$ and $\delta_{trans},$ since we assume that $R$ is sufficiently large, we find that the first factor in parentheses  is bounded by $1,$ and thus
$$
\int_{B_R}(Br_{2\alpha}\E f)^{3.25} \le
C_\epsilon R^\epsilon\bigg(\sum_\tau\int|f_\tau|^2\bigg)^{3/2+\epsilon}
R^{\delta_{trans}\log(K^\epsilon\alpha\mu)}.
$$
This closes the induction and thus completes the proof of Theorem \ref{largetheorem}.

\medskip

\subsection{Proof of the Geometric Lemma}\label{proofgeom}
In this subsection we prove Lemma \ref{geometric}.
Assume that we are given a  family of caps $\{\tau_k\}$ such that for any $k,m$ with  $k\ne m$ we have
\begin{equation}\label{notstrong}
\min\{|y^c_m-y^c_k|, \max\{|{\bf t}^\gamma_{z^c_m}(z^c_m,z^c_k)|,\,|{\bf t}^\gamma_{z^c_k}(z^c_m,z^c_k)|\}\}\le 10 \mu^{1/2}K^{-1},
\end{equation}
where we denote by $z^c_k=(x^c_k,y^c_k)$ the center of the cap $\tau_k.$
\smallskip

\noindent\bf Case 1\rm. For all $k,m$ we have $|y^c_m-y^c_k|\le {10\mu^{1/2}}K^{-1}.$ Then, all caps are
contained in a
horizontal strip of width ${10\mu^{1/2}}K^{-1}\le \mu^{1/2}K^{-1/4},$ since $K\ge 30.$
\smallskip

\noindent\bf Case 2\rm. There are two caps, say  $\tau_1,\,\tau_2,$ such that $|y^c_1-y^c_2|>{10\mu^{1/2}}K^{-1}.$
We may assume that
$y^c_2-y^c_1=\max_{j\ne k}|y^c_j-y^c_k|.$
Then, for all $k,$
$$y^c_1\le y^c_k\le y^c_2.$$
Since $\tau_1$ and $\tau_2$ are not strongly separated, 
$|{\bf t}_{z^c_1}^\gamma(z^c_1,z^c_2)|\le {10\mu^{1/2}}K^{-1}$ and 
$|{\bf t}_{z^c_2}^\gamma(z^c_1,z^c_2)|\le{10\mu^{1/2}}K^{-1}.$ Therefore, by \eqref{TV3}, 

$$
|\gamma |\,|y^c_2-y^c_1|^2=| {\bf t}_{z^c_1}^\gamma(z^c_1,z^c_2)-{\bf t}_{z^c_2}^\gamma(z^c_1,z^c_2)|
\le 20\mu^{1/2}K^{-1},
$$
and since $|y^c_k-y^c_1|\le  |y^c_2-y^c_1|,$ we see that   for all $k$ 
\begin{equation}\label{dist1}
 |y^c_k-y^c_1|\le (20 \mu^{1/2}|\gamma|^{-1}K^{-1})^{\frac 12}.
\end{equation}
In combination with  \eqref{TV3}, this also  implies that 
\begin{eqnarray}\nonumber
|x^c_k-x^c_1|&\le& |{\bf t}_{z^c_1}^\gamma(z^c_1,z^c_k)|+|\gamma| |y^c_k|\, |y^c_k-y^c_1|\\ \nonumber
&\le& {10\mu^{1/2}}K^{-1}+|\gamma|(20\mu^{1/2}|\gamma|^{-1}K^{-1})^{\frac 12}\\ 
 &\le& 15\mu^{1/2}K^{-1/2},  \label{dist2}
\end{eqnarray}
since $|\gamma|\le1,$ $\mu\ge1$ and $K\ge1.$
\smallskip

When $|\gamma| K^{1/2} >1,$ we conclude from \eqref{dist1} that $|y^c_k-y^c_1|\le (20 \mu^{1/2}K^{-1/2})^{1/2}$  for all $k.$ Hence, all the caps are contained in a horizontal 
strip of width $10 \mu^{1/2} K^{-1/4}. $
Decomposing these further into horizontal strips of width $0.5\mu^{1/2}K^{-1/4},$  each of which is contained in one of the ragged strips $S_\ell$ that have been fixed in Section \ref{reduction}, and distributing the caps $\tau_k$ of our family  over these ragged strips, we arrive at at most 40 horizontal ragged strips of width  $\mu^{1/2}K^{-1/4}$ which contain all the  caps considered in Case 2. Note that by our passage to ragged strips the width does not increase by more than  $2\mu^{1/2} K^{-1}\le 0.5\mu^{1/2}K^{-1/4},$ since $K\ge 20.$

\smallskip

When $|\gamma| K^{1/2} \le 1,$ we conclude from \eqref{dist2} that all the caps are contained in a vertical  strip of width $20\mu^{1/2}K^{-1/2},$  and arguing as before we can conclude the proof of Lemma \ref{geometric} also in this case. 

\qed

\section{Passing from extension estimates on cubes to estimates on plates: an othogonality lemma}
We will here finally discuss an auxilary lemma that we needed in Section \ref{broadpoints}.\\
Let $\Omega\subset\R^n$ be an open bounded set and  $\phi:\Omega\to\R$ any phase function such that $|\nabla\phi(x)|\lesssim 1$ for all $x\in\Omega$. Assume further that  $\rho\in C^\infty(\Omega),$ and consider the Fourier extension operator
$$
\mathcal{E} f(\xi):=\int_\Omega f(x)\, e^{-i[\xi'\cdot x +\xi_{n+1} \phi(x)]} \rho(x)\, dx,
$$
where $\xi=(\xi',\xi_{n+1})\in\R^n\times\R$ (for convenience, we have chosen here a different sign in the phase than in the definition of $\E f$).
\begin{lemnr}\label{aux}
 Let $2\leq q\leq p$. Assume that  for every  $\epsilon>0$ there exists a constant $C_\epsilon$ such that for every $R\geq1$
\begin{align}\label{eins}
	\|\mathcal{E} f\|_{L^p([0,R]^{n+1})} \leq C_\epsilon R^{\epsilon}
	\|f\|_{L^2(\Omega )}^{2/q}\, \|f\|_{L^\infty(\Omega)}^{1-2/q}
\end{align}
for all $f\in L^q(\Omega)$. Then for every  $\epsilon>0$ there exista a constant $C'_\epsilon$ such that for all $R\geq1$
\begin{align}\label{zwei}
	\|\mathcal{E} f\|_{L^p(\R^n\times[0,R])} \leq C'_\epsilon R^{\epsilon}
	\|f\|_{L^2(\Omega )}^{2/q}\, \|f\|_{L^\infty(\Omega)}^{1-2/q}
\end{align}
for all $f\in L^q(\Omega)$.
\end{lemnr}
\begin{proof}
First observe that \eqref{eins} holds for any translate of $[0,R]^{n+1}$ in place of $[0,R]^{n+1}$ as well, in particular on  any cube $Q_y:= R(y,0)+[0,R]^{n+1}, \, y\in\Z^n$.

In order to pass to a corresponding estimate on the plate $\R^n\times[0,R]),$ which decomposes into the cubes $Q_y,$ it will suffice to perform an adapted  frequency decomposition of $f$ (a full wave packet  decomposition is needed here):

Let $f_y := f \ast\check\chi_y$, where $\chi_y(\eta):=\chi(\eta/R-y)$, $y\in\Z^n,$ and $\chi$ is a  suitable compactly supported bump function chosen so that the  $\chi_y, y\in \Z^n,$ form a partition of unity on $\R^n.$ Then  $f=\sum_y f_y$.

In order to prove \eqref{zwei},  we may and shall assume that  $0\leq\xi_{n+1}\leq R$. Under this restriction, we will see that $\mathcal{E} f_y$ is essentially supported in $Q_y$. Indeed, note that by Fourier inversion
\begin{align*}
	\mathcal{E} f_y(\xi) =c_n \iint\hat f(\eta)\chi(\eta/R-y)e^{-i[(\xi'-\eta)\cdot x + \xi_{n+1} \phi(x)]} \rho(x)dxd\eta.
\end{align*}
The gradient in $x$ of the full phase is given by  $\xi'-\eta+\xi_{n+1}\nabla\phi(x)=\xi'-Ry+O(R)$.
Hence, by a standard integration by parts argument in $x,$  we see that for any $N\in\N$
\begin{align}\label{decay}
	|\mathcal{E} f_y(\xi)| \lesssim_N R^n\big|\xi'-Ry\big|^{-N} \|f_y\|_1, \qquad \text{if}\ \big|\xi'-Ry\big|\gg R.
\end{align}
It is thus natural to split (and estimate, using Minkowski's inequality)
\begin{align*}
	\|\mathcal{E} f\|_{L^p(\R^n\times[0,R])}
	=& \left(\sum_{y} \|\sum_{z}\mathcal{E} f_{y+z}\|_{L^p(Q_{y})}^p\right)^{1/p}
	\leq \sum_{z}\left(\sum_{y} \|\mathcal{E} f_{y}\|_{L^p(Q_{y-z})}^p\right)^{1/p}\\
\end{align*}
into two parts: First we use \eqref{eins} and $p\geq q$ to estimate
\begin{align*}
 \sum_{|z|\lesssim 1}\left(\sum_{y} \|\mathcal{E} f_{y}\|_{L^p(Q_{y-z})}^p\right)^{1/p}
			\lesssim &  C_\epsilon R^{\epsilon} \left(\sum_{y}
	\|f_y\|_{L^2(\Omega )}^{2p/q}\, \|f_y\|_{L^\infty(\Omega)}^{p(1-2/q)}\right)^{1/p}\\
	\leq& C_\epsilon R^{\epsilon} \left(\sum_{y} \|f_y\|_{L^2(\Omega )}^{2}\right)^{1/q} \|f\|_{L^\infty(\Omega)}^{1-2/q}\\
	\lesssim& C_\epsilon R^{\epsilon} \|f\|_{L^2(\Omega)}^{2/q} \|f\|_{L^\infty(\Omega)}^{1-2/q},	 
\end{align*}
where have used  Plancherel's theorem.
The remainder can be estimated using \eqref{decay}:
\begin{align*}
 \sum_{|z|\gg 1}\left(\sum_{y} \|\mathcal{E} f_{y}\|_{L^p(Q_{y-z})}^p\right)^{1/p}
 	\lesssim& R^{n+(n+1)/p}\sum_{|z|\gg 1} (R|z|)^{-N}\left(\sum_{y} \|f_{y}\|_1^p\right)^{1/p}\\
 	\lesssim& R^{-N'} \left(\sum_{y} \|f_{y}\|_1^p\right)^{1/p},
\end{align*}
which finishes the proof because $\|f_{y}\|_1\leq \|f_{y}\|_{2}^{2/q} \|f_{y}\|_{\infty}^{1-2/q},$ so that  from here we can proceed as before.

\end{proof}

\bigskip

\thispagestyle{empty}

\renewcommand{\refname}{References}

\end{document}